\documentclass{amsart}
\usepackage{amsmath,amsthm,latexsym,amssymb,hyperref, amsfonts}
\usepackage{stmaryrd}
\usepackage{eucal}
\usepackage{mathrsfs}
\usepackage[all]{xy}
\usepackage{color}
\usepackage{tikz}
\usepackage{url}
\usepackage{comment}
\usepackage{tikz-cd}
\usepackage{enumerate}
\usepackage{enumitem}
\usepackage{dutchcal}
\usepackage{hyperref}

\usepackage[normalem]{ulem}
\usetikzlibrary{matrix,arrows,decorations.pathmorphing}

\numberwithin{equation}{section}

\theoremstyle{definition}
\newtheorem{defi}{Definition}[section]

\newtheorem{ex}[defi]{Example}
\newtheorem{rem}[defi]{Remark}
\newtheorem{notation}[defi]{Notation}
\newtheorem*{ack}{Acknowledgements}
\newtheorem{cond}{Condition}

\theoremstyle{plain}
\newtheorem{thm}[defi]{Theorem}
\newtheorem{prop}[defi]{Proposition}
\newtheorem{lem}[defi]{Lemma}
\newtheorem{cor}[defi]{Corollary}

\newcommand{\thh}{\mathsf{THH}}
\newcommand{\cothh}{\mathsf{coTHH}}
\newcommand{\cohh}{\mathsf{coHH}}
\newcommand{\Sp}{\mathbb{S}}
\newcommand{\sph}{\mathbb{S}}
\newcommand{\ei}{{\mathbb{E}_\infty}}
\newcommand{\ee}{{\mathbb{A}_\infty}}

\newcommand{\ot}{\otimes}
\newcommand{\op}{{^\mathsf{op}}}
\newcommand{\Fin}{\mathsf{Fin}}
\newcommand{\rom}[1]{\uppercase\expandafter{\romannumeral #1\relax}}

\newcommand{\Funlax}{\mathsf{Fun}^\mathsf{lax}}

\newcommand{\Funcolax}{\mathsf{Fun}^\mathsf{colax}}

\newcommand{\C}{\mathcal{C}}
\newcommand{\D}{\mathcal{D}}

\newcommand{\Co}{{\mathcal{C}^\ot}}

\renewcommand{\O}{\mathcal{O}}
\newcommand{\Oo}{{\mathcal{O}^\ot}}

\newcommand{\alg}{\mathsf{Alg}}
\newcommand{\coalg}{\mathsf{CoAlg}}

\newcommand{\Mod}{\mathsf{Mod}}

\newcommand{\THH}{\mathsf{THH}}
\newcommand{\coTHH}{\mathsf{coTHH}}
\newcommand{\HH}{\mathsf{HH}}
\newcommand{\Der}{\mathsf{Der}}
\newcommand{\Tor}{\mathsf{Tor}}
\newcommand{\Ext}{\mathsf{Ext}}
\newcommand{\Hom}{\mathsf{Hom}}

\newcommand{\Fun}{\mathsf{Fun}}

\newcommand{\assoc}{\mathsf{Assoc}}
\newcommand{\com}{\mathsf{Com}}

\newcommand{\XXX}{weak quasi}

\newcommand{\dual}{\mathsf{fd}}
\newcommand{\qdual}{\star}
\newcommand{\xdual}{\circ}

\newcommand{\id}{\mathsf{id}}

\renewcommand{\bar}{\mathsf{Bar}_\bullet}
\newcommand{\cobar}{\mathsf{coBar}^\bullet}

\newcommand{\M}{\mathsf{M}}
\newcommand{\MM}{{N(\M_c)\left[\Wo\right]}}

\renewcommand{\sp}{\mathsf{Sp}}

\newcommand{\W}{\mathsf{W}}
\newcommand{\Wo}{\mathsf{W}^{-1}}
\newcommand{\Ww}{{\mathsf{W}'}}
\newcommand{\Wwo}{{{\mathsf{W}'}^{-1}}}

\newcommand{\RomanNumeralCaps}[1]
    {\MakeUppercase{\romannumeral #1}}
\newcommand{\hfp}{H\mathbb{F}_p}
\newcommand{\I}{\mathbb{I}}
\newcommand{\z}{\mathbb{Z}}
\newcommand{\fp}{\mathbb{F}_p}

\newcommand{\hz}{H\mathbb{Z}}
\newcommand{\lv}{\lvert}
\newcommand{\rv}{\rvert}
\newcommand{\lambdafp}{\Lambda_{\mathbb{F}_p}}
\newcommand{\wdg}{\wedge}
\newcommand{\wdgz}{\wedge_{H\mathbb{Z}}}

\long\def\cyantext#1{{\color{cyan}#1}}

\long\def\emptytext#1{}
\def\co{\colon\thinspace}
\title[Spanier-Whitehead duality for coTHH]{Spanier-Whitehead duality for Topological coHochschild homology}

\author{Haldun Özgür Bayındır}
\address{Department of Mathematics, City, University of London, Northampton Square, London EC1V 0HB, UK}
\email{ozgurbayindir@gmail.com}

 \author{Maximilien P\'eroux}
\address{Department of Mathematics, University of Pennsylvania,
209 South 33rd Street,
Philadelphia, PA, 19104-6395, USA}
\email{mperoux@sas.upenn.edu}

    \makeatletter
\@namedef{subjclassname@2020}{\textup{2020} Mathematics Subject Classification}
\makeatother


\begin{document}



\begin{abstract} 

In this work, we compute the topological coHochschild homology (coTHH) of interesting coalgebras such as the Steenrod algebra spectrum. For this, we  start by extending the Hess-Shipley definition of coTHH  to $\infty$-categories, following the Nikolaus-Scholze approach to THH. Furthermore, we prove that coTHH of what we call quasi-proper coalgebras can be obtained from THH via Spanier-Whitehead duality which provides further insight into coTHH and its relationship to THH. 
\end{abstract}

\maketitle

\section{Introduction}

\subsection*{Background} 

Topological Hochschild homology (THH) for ring spectra extends the notion of Hochschild homology for algebras. The topological persepctive provides new insight on rings.
For instance, the Hochschild homology of the finite field $\mathbb{F}_p$ is a divided power algebra, but  B\"okstedt shows that the topological Hochschild homology of its Eilenberg-Mac Lane spectrum $H\mathbb{F}_p$ is a polynomial ring. More recent examples can be found in the work of \cite{bms}.
THH plays also an important role as an invariant for algebras in spectra thanks to its connections with $K$-theory via the Dennis trace map. Furthermore, THH is of interest for string topology. For a connected topological space $X$, there is an equivalence
\[\thh(\Sigma^\infty_+ \Omega X) \simeq \Sigma^\infty_+ \mathscr{L}X,\]
where $\mathscr{L}X$ denotes the free loop space on $X$. This goes back to \cite[Theorem 7.3.11]{loday1998cyclichomology}; see  \cite[IV.3.2, IV.3.3]{tch} for a modern reference.

Topological coHochschild homology (coTHH) is a homology theory for coalgebras in spectra introduced by Hess-Shipley in \cite{HScothh}, analogous to THH for ring spectra. 
Furthermore, coTHH extends the notion of  coHochschild homology (coHH) introduced in \cite{doi} and \cite{cohh} to coalgebras in a symmetric monoidal model category \cite[2.3]{toolscothh}. 
One reason why this invariant of coalgebras is important  is because  coalgebraic structures are shown to contain valuable information regarding topological spaces. For instance, \cite{rivera} establish that homotopy type of a topological space is entirely captured by the coalgebraic structure of its singular chain complex. Furthermore, Hess and Shipley prove that Waldhausen $K$-theory of a topological space $X$ can be obtained using the category of comodules over the coalgebra $\Sigma^\infty_+ X$ \cite{comodHS}. 

Just like THH, topological coHochschild homology also provides a new model for free loop spaces. Namely, Hess and Shipley show that there is an equivalence
\[\cothh(\Sigma^{\infty}_+ X) \simeq \Sigma^\infty_+ \mathscr{L}X\]
for every ``EMSS-good space" $X$; in particular, for every simply connected space $X$  \cite{HScothh}.
Connecting with the result above, for every EMSS-good space $X$, one obtains an instance of Koszul duality between THH and coTHH:
\[
\thh(\Sigma^\infty_+ \Omega X) \simeq \cothh(\Sigma^\infty_+ X).
\]
In \cite[2.3]{toolscothh}, the authors construct a coB\"okstedt spectral sequence for coTHH and show that this spectral sequence has a ``$\square$-Hopf algebra structure", i.e.\ it is endowed with a comultiplicative structure that is compatible with its multiplicative structure. The $E_2$-page of this spectral sequence is computed for various interesting coalgebra spectra in \cite{sarah}. In \cite{bohmann2022topological}, Bohmann, Gerhardt and Shipley use the $\square$-Hopf algebra structure on the coB\"okstedt spectral sequence and the equivalence above to obtain homology computations for various free loop spaces  generalizing earlier computations of \cite{kuribayashi1997cohomologyoffreeloopspaces}. 
Recently, Ayala and Francis defined  factorization cohomology which generalizes coTHH for $\mathbb{E}_n$-coalgebras to obtain a Poincar\'e duality result for factorization homology \cite{zeromani}. 
From work in progress, the second author and Klanderman provide a shadow structure in the sense of \cite{ponto2008fixed} for coHH, resulting in new interesting bicategorical traces for coHH \cite{tracecohh}.

\subsection*{Computations}
In this work, we compute the topological coHochschild homology of new coalgebras not considered before. For this, we carry the definition of coTHH to the $\infty$-categorical setting and we prove a duality relationship between coTHH and THH which provides further insight to coTHH, and also to THH.

Our main computation is the  coHochschild homology groups of the Steenrod algebra spectrum ${[\hfp,\hfp]}_{\sph}$. Here ${[\hfp,\hfp]}_{\sph}$ denotes the spectrum of $\sph$-module endomorphisms of $\hfp$. Firstly, we show that ${[\hfp,\hfp]}_{\sph}$ inherits the structure of an $\ei$-$\hfp$-coalgebra, i.e.\ an $\ei$-coalgebra in $\hfp$-modules, as it is the linear dual of the $\ei$-$\hfp$-algebra $\hfp \wdg \hfp$. Note that this does not follow by the correspondence  between $\ei$-algebras and $\ei$-coalgebras on dualizable objects (Proposition \ref{prop: antiequivalence between dualizable objects}), given by the dualization functor. This is because ${[\hfp,\hfp]}_{\sph}$ is not dualizable in $\hfp$-modules as the Steenrod algebra 
\[\mathcal{A}\cong \pi_*\left({[\hfp,\hfp]}_{\sph}\right)\]
is infinite dimensional, see Example \ref{ex dualizable in hk modules}. On the other hand, our   Theorem \ref{thm equivalence of connective coalgebras and coconnective coalgebras in chains}, which generalizes the correspondence  on dualizable objects, does equip ${[\hfp,\hfp]}_{\sph}$ with the structure of an $\ei$-$\hfp$-coalgebra coming  from the $\ei$-$\hfp$-algebra structure of $\hfp \wdg \hfp$. Using  our new duality result for coTHH (Theorem \ref{thm cothh duality in chains}), together with B\"okstedt periodicity, we compute the coHochschild homology groups of ${[\hfp,\hfp]}_{\sph}$. 

\begin{thm}[Theorem \ref{thm: cothh of steenrod algebra}]
There is an equivalence of graded $\fp$-modules:
\[
\pi_*\left(\cothh^{\hfp}\left({[\hfp,\hfp]}_{\sph}\right)\right) \cong \mathcal{A} \otimes \fp[x_{-2}],\]
where $\lv x_{-2}\rv = -2$.
\end{thm}

Moreover, using Theorem \ref{thm cothh duality in chains}, we do the following computations.

\begin{itemize}
    \item Theorem \ref{thm cothh of dsa in chains}: there is an $\ei$-coalgebra structure on $\hfp \wdgz \hfp$ as an $\hfp$-module, and its coHochschild homology is given by:
  \[\pi_*\left(\cothh^{\hfp}(\hfp \wdgz \hfp)\right)\cong \lambdafp(x_1) \otimes \fp[[t]],\]
where $\lv t \rv = 0$ and $\lv x_1 \rv = 1$. This completes the result of \cite{sarah} in which the $E_2$-page of the relative coB\"okstedt spectral sequence from \cite{toolscothh} computing the homotopy groups above was calculated.

\item Theorem \ref{thm: loop-hfp}: there is an $\ee$-coalgebra structure on $\Omega\hfp$ as an $H\z$-module and its topological coHochschild homology is given by:
\[\pi_*\left(\cothh^{H\z}(\Omega H\fp)\right)\cong \Omega \fp[x_{-2}],\]
where $\lv x_{-2} \rv= -2$ and $\Omega$ on the right hand side denotes the functor that decreases the grading by 1.
\end{itemize}

\subsection*{From model categories to higher categories}
We mentioned earlier the importance of coalgebraic structures in spectra. However, there is no model category of spectra that accurately captures the homotopy theory of its coalgebras, see \cite{perouxshipley} and \cite[5.6]{perouxDKloc}. 
An $\ee$-coalgebra in spectra is a spectrum endowed with a comultiplication that is coassociative and counital up to higher homotopy. 
Given an $\ee$-coalgebra, there is not, in general, a choice of monoidal model category of spectra in which the $\ee$-coalgebra structure can be rigidified to a strictly coassociative counital coalgebra structure.
The major disadvantage of model categories is that the Spanier-Whitehead duality, that relates $\ee$-algebras with $\ee$-coalgebras in finite spectra (see \cite[3.2.5]{lurie2}), cannot be realized at the level of strict algebras and coalgebras. Essentially, model categories currently can only capture coalgebras of the form $\Sigma^\infty_+ X$ induced by the diagonal on the space $X$. It is for this very reason that the work of \cite{HScothh}, \cite{toolscothh}, and \cite{sarah} is limited to the computation of $\cothh(\Sigma^\infty_+ X)$.

 If a symmetric monoidal $\infty$-category $\C$ is endowed with an internal hom $[-,-]$, we can define the notion of linear duality that generalizes the classical dual of a vector space over a field, or the Spanier-Whitehead dual of a spectrum. For $X$ an object in $\C$ we denote $X^\vee=[X, \I]$ the linear dual of $X$, where $\I$ denotes the monoidal unit of $\C$.
If $C$ is a coalgebra in $\C$, then $C^\vee$ is always an algebra in $\C$.
It is well-known that algebras and coalgebras in finite dimensional vector spaces are anti-equivalent, see \cite[I.1.1]{sweedler}. Similarly, by \cite[3.2.5]{lurie2} (see also Corollary \ref{cor: antiequivalence proper (co)algebras} below), the linear dual provides an anti-equivalence between $\ee$-algebras and $\ee$-coalgebras in dualizable objects in $\C$. 

 To be able to incorporate Spanier-Whitehead duality into our computational framework, we extend the definition of coTHH to any symmetric monoidal $\infty$-category $\C$, following the THH approach of \cite{tch}. Indeed, our computations provide an instance demonstrating the strength of $\infty$-categories for computational purposes.

\subsection*{A duality relationship between coTHH and THH} Using the anti-equiva\-lence between $\ee$-algebras and $\ee$-coalgebras in dualizable objects mentioned above, and our  $\infty$-categorical definition of coTHH, we prove the following theorem that relates THH and coTHH of algebras and coalgebras in dualizable objects via the linear dual. 

\begin{thm}[{Theorem \ref{thm most general thh cothh duality}}] \label{thm cothh duality general}
Let $R$ be an $\ei$-ring spectrum.
If $C$ is an $\ee$-coalgebra over $R$, whose underlying $R$-module is dualizable, then there is an equivalence of $R$-module spectra
\[\cothh^{R}(C) \simeq \left(\thh^R(C^\vee)\right)^\vee.\]
\end{thm}

We in fact generalize the anti-equivalence on dualizable objects to larger subclasses of algebras and coalgebras where the dualization functor is strong monoidal, see Theorem \ref{thm: anti-equivalence between quasi-proper (co)algebras}. We also extend the duality relationship  between THH and coTHH to these new subclasses. The main ingredient for these generalizations is our notions of \emph{quasi-dualizability} (Definition \ref{def: quasi-dualizable}) and \emph{quasi-proper} coalgebras (Definition \ref{def quasi proper}) which are weaker conditions than dualizability, see Remark \ref{rem: different notions of dualizable}. We show in Theorem \ref{thm most general thh cothh duality} that the equivalence in Theorem \ref{thm cothh duality general} also holds for quasi-proper $\ee$-coalgebras. 


With specific examples in mind, such as the Steenrod algebra spectrum mentioned earlier, we are led to consider a more general class of $HA$-coalgebras than those whose underlying $HA$-modules are  dualizable. Here $A$ denotes a discrete commutative ring.
 Through careful Tor and Ext spectral sequence considerations, we are able to show that these  $HA$-coalgebras are indeed quasi-proper, see Section \ref{sec proof of cothh thh duality in chains}. This provides Theorem \ref{thm cothh duality in chains} below, which extends Theorem \ref{thm cothh duality general} for $HA$-coalgebras. Indeed, this generalization turns out to be very useful for computations. For instance, our computation of the coTHH of the Steenrod algebra spectrum relies on this generalization.

We say an $HA$-module $X$ is of \emph{finite type} if $\pi_iX$ is a finitely generated $A$-module for each $i$. Recall that a ring $A$ is said to have finite global dimension if there is an integer $d$ such that  every $A$-module $M$ admits a projective resolution of length at most $d$.




\begin{thm}[{Theorem \ref{thm restatement of cothh thh duality in chains}}]\label{thm cothh duality in chains}
Let $A$ be a discrete commutative Noetherian ring with finite global dimension and let $C$ be a  connective or coconnective $\ee$-coalgebra over $HA$ of finite type. Then there is an equivalence of $HA$-module spectra
\[\cothh^{HA}(C) \simeq \left(\thh^{HA}(C^\vee)\right)^\vee.\]
\end{thm}
We obtain the following new insight on THH and coTHH.
\begin{itemize}
   
\item Example \ref{ex: hess-shipley}: The suspension spectrum of a finite CW-complex $X$ is a dualizable object in spectra. When $X$ is simply connected, one has
\[
\left(\thh((\Sigma^\infty_+ X)^\vee)\right)^\vee\simeq \Sigma^\infty_+ \mathscr{L}X,
\]
due to \cite{kuhn} and \cite{cary}. We obtain a new proof of this equality using Theorem \ref{thm cothh duality general} and the results of \cite{HScothh}. Indeed, this generalizes the equality above to EMSS-good finite CW-complexes. Furthermore, this provides the following Koszul duality relationship for THH:
\[\left(\thh((\Sigma^\infty_+ X)^\vee)\right)^\vee \simeq \thh(\Sigma^\infty_+ \Omega X).\]

\item Example \ref{ex: compact Lie group}: if $G$ is a compact Lie group, then the Thom spectrum $G^{-\tau}$ of its stable normal bundle is endowed with an $\ee$-coalgebra structure in spectra, induced by the group multiplication of $G$, and its coTHH is given by:
\[
\cothh(G^{-\tau})\simeq \left(\Sigma^\infty_+ \mathscr{L}BG\right)^\vee, 
\]
where $\mathscr{L} BG$ is the free loop space of the classifying space of $G$. By \cite{perouxDKloc}, such a computation could not have been considered in the framework of \cite{HScothh}.

\end{itemize}
\subsection*{Outline} In Section \ref{sec coTHH definition}, we define coTHH in general symmetric monoidal  $\infty$-cate\-gories. After this, we discuss duality between coalgebras and algebras and define our notions of quasi-dualizability and quasi-properness in Section \ref{sec general duality between algebras and coalgebras}. Using this, we prove general duality results between coTHH and THH in Section \ref{sec general duality between thh and cothh}. In Section \ref{sec proof of cothh thh duality in chains}, we apply these duality results to $HA$-module spectra and prove Theorems \ref{thm cothh duality in chains} and \ref{thm equivalence of connective coalgebras and coconnective coalgebras in chains}. Section \ref{sec examples of coalgebras in spectra} is devoted to a discussion of examples of $\ee$-coalgebras in spectra; we show that many examples of spectra such as $ku$, $MU$ and $ko$ etc.\ are not $\ee$-coalgebras in spectra. In Section \ref{sec computations in ha modules}, we do  explicit coHochschild homology computations using Theorem \ref{thm cothh duality in chains}. The arguments in this section only make use of Theorems \ref{thm cothh duality in chains} and \ref{thm equivalence of connective coalgebras and coconnective coalgebras in chains}; therefore, the reader interested in these computations may jump to Section \ref{sec computations in ha modules}. 

\begin{ack} We thank Sarah Klanderman for initial conversations with her that sparked some ideas in this project. {We are grateful to Elden Elmanto and Thomas Nikolaus for their careful reading of this work.} We also thank Rune Haugseng and Maxime Ramzi for helpful conversations. The first author acknowledges support from the project ANR-16-CE40-0003 ChroK. 
\end{ack}

\begin{notation} We begin by setting notation that we  use throughout and recalling some elementary notions of the theory of $\infty$-categories, following \cite{htt, HA}. The notions of \emph{symmetric monoidal $\infty$-categories} and \emph{$\infty$-operads} are defined respec\-tively in \cite[2.0.0.7, 2.1.1.10]{HA}.

\begin{enumerate}
\item We denote by $\assoc^\ot$ the \emph{associative $\infty$-operad} as in \cite[4.1.1.3]{HA}. If $\C$ is a monoidal $\infty$-category, we denote the $\infty$-category of $\ee$-algebras in $\C$, as described in \cite[4.1.1.6]{HA}, by $\alg_\ee(\C)$. Recall there is an approximation of $\infty$-operads $N(\Delta\op)\rightarrow \assoc^\ot$, see \cite[4.1.2.11]{HA}. Here $\Delta$ denotes the usual simplex category and $N$ is the nerve of a category.

\item We denote by $\com^\ot=N(\Fin_*)$ the \emph{commutative $\infty$-operad} as in \cite[2.1.1.18]{HA}. Here $\Fin_*$ is the category of finite pointed sets. If $\C$ is a symmetric monoidal $\infty$-category, we denote the $\infty$-category of $\ei$-algebras in $\C$, as described in \cite[2.1.3.1]{HA}, by $\alg_{\ei}(\C).$

\item Let $\O$  be an $\infty$-operad. Let $\C$ and $\D$ be $\O$-monoidal $\infty$-categories, as defined in \cite[2.1.2.13]{HA}. We say a functor $F:\C\rightarrow \D$ is \emph{lax $\O$-monoidal} if it is a map of $\infty$-operads in the sense of \cite[2.1.2.13]{HA}.
We denote the $\infty$-category of lax $\O$-monoidal functor from $\C$ to $\D$ by $\Funlax_\O(\C, \D)$. 
We say $F$ is \emph{strong $\O$-monoidal} if it is $\O$-monoidal in the sense of \cite[2.1.3.7]{HA}. 
When $\O^\ot=\com^\ot$, we prefer to say \emph{symmetric} monoidal instead of $\O$-monoidal. 

\item Given a symmetric monoidal $\infty$-category $\C$ and an $\ei$-algebra $A$ in $\C$, we denote by $\Mod_A(\C)$ the $\infty$-category of \emph{left} modules over $A$ in $\C$ as in \cite[4.2.1.13]{HA}. As $A$ is an $\ei$-algebra, the $\infty$-category is also equivalent to the $\infty$-category of right modules over $A$ in $\C$ by \cite[4.5.1.6]{HA}.

\item We denote the $\infty$-category of spectra by $\sp$ as in \cite[1.4.3.1]{HA}. By a \emph{commutative ring spectrum} we mean an $\ei$-algebra in $\sp$. We denote by $\sph$ the \emph{sphere spectrum}. If $A$ is a discrete commutative ring, we denote by $HA$ its associated \emph{Eilenberg-Mac Lane spectrum}.

\item Given a commutative ring spectrum $R$, by an \emph{$R$-algebra} we mean an $\ee$-algebra in $\Mod_R(\sp)$. By a \emph{commutative $R$-algebra}, we mean an $\ei$-algebra in $\Mod_R(\sp)$. See also Definition \ref{defi: R-coalgebras in spectra} below for the dual case of coalgebras.
\item For a discrete commutative ring $A$, an $A$-differential graded algebra is a monoid object in the symmetric monoidal category of chain complexes in $A$-modules. We call these $A$-DGAs. 
\end{enumerate}
\end{notation}

\section{Topological CoHochschild Homology}\label{sec coTHH definition}

We provide here, in Definition \ref{defi: cothh}, the notion of topological coHochschild homology for an $\ee$-coalgebra in a general symmetric monoidal $\infty$-category. We show that it is equivalent to the model categorical approach of \cite{HScothh} and \cite{toolscothh} in Proposition \ref{Prop: def of cothh agrees}. 

Cocommutative coalgebras in $\infty$-categories were introduced in \cite{lurie2}, and a more general definition of coalgebras was given in \cite{coalgenr}.
Essentially, coalgebras are algebras in the opposite category. If $\O$ is an $\infty$-operad, and $\C$ is {an} $\O$-monoidal $\infty$-category, as in \cite[2.1.1.18]{HA}, then we need to describe the $\O$-monoidal structure on its opposite category $\C\op$. The case $\O^\ot=\com^\ot$ is done in \cite[2.4.2.7]{HA}, and it can be generalized using \cite{dualcocart}.

\begin{defi}[{\cite[2.1]{coalgenr}}]\label{defi: coalgebras}
Let $\O$ be an $\infty$-operad. Let $\C$ be an $\O$-monoidal $\infty$-category. The $\infty$-category of $\O$-coalgebras in $\C$ is defined as:
\[
\coalg_\O(\C)=\left(\alg_\O(\C\op)\right)\op.
\]
We shall mostly be interested in the case where $\Oo=\assoc^\ot$ or $\Oo=\com^\ot$. 
\end{defi}

\begin{rem}
By \cite[2.3]{coalgenr}, if $\C$ is the nerve of a symmetric monoidal (ordinary) category $\M$, then the $\infty$-category $\coalg_\ee(\C)$ corresponds precisely to the nerve of the category of (strictly) coassociative counital coalgebras in $\M$. Similarly, the $\infty$- category $\coalg_\ei(\C)$ corresponds precisely to the nerve of the category of (strictly) cocommutative counital coalgebras in $\M$.
\end{rem}

\begin{defi}\label{defi: R-coalgebras in spectra}
Let $R$ be a commutative ring spectrum. An \emph{$R$-coalgebra} is an $\ee$-coalgebra in $\Mod_R(\sp)$. A \emph{cocommutative $R$-coalgebra} is an $\ei$-coalgebra in $\Mod_R(\sp)$.
\end{defi}

Recall that an $\O$-algebra in $\C$ is simply a lax $\O$-monoidal functor $\O\rightarrow \C$. In other words, we have the equivalence $\alg_\O(\C)\simeq\Funlax_\O(\O, \C)$. We have a dual characterization for coalgebras. 
\begin{defi}
Let $\O$ be an $\infty$-operad. Let $\C$ and $\D$ be $\O$-monoidal $\infty$-categories. We say that $F:\C\rightarrow \D$ is \emph{colax $\O$-monoidal} (sometimes also called oplax $\O$-monoidal) if its opposite $F\op:\C\op\rightarrow \D\op$ is lax $\O$-monoidal. We denote the category of colax $\O$-monoidal functor by $\Funcolax_\O(\C, \D)=(\Funlax_\O(\C\op, \D\op))\op$.
\end{defi}

\begin{prop}\label{prop: colax=coalgebra}
Let $\O$ be an $\infty$-operad. Let $\C$ be an $\O$-monoidal category. Then we have an equivalence:
\[
\coalg_\O(\C)\simeq \Funcolax_\O(\O, \C).
\]
\end{prop}

\begin{proof}
Recall from \cite[2.4.2.7]{HA} and \cite{dualcocart} that $\C\op$ is $\O$-monoidal as follows. Given a coCartesian fibration $p:\C^\otimes\rightarrow \O^\otimes$ that provides $\C$ its $\O$-monoidal structure, its straightening is a functor $F:\O^\otimes \rightarrow \widehat{\mathsf{Cat}}_\infty$, with value in the $\infty$-category of not necessarily small $\infty$-categories, as in \cite[3.0.0.5]{HA}. The functor $F$ also classifies a Cartesian fibration $\widehat{p}:\widehat{\C^\otimes}\rightarrow (\mathcal{O}^{\otimes})\op$. 
The fiber of $X\in \mathcal{O}$ over $\widehat{p}$ is equivalent to $(\C_X)$ -- the fiber of $X$ over $p$.
Taking the opposite yields a coCartesian fibration ${\widehat{p}}\op:(\widehat{\C^\otimes})\op \rightarrow \O^\otimes$. 
This coCartesian fibration  provides a $\O$-monoidal structure on $\C\op$. 
Notice moreover that if we view $\O$ as a $\O$-monoidal category, then its $\O$-monoidal structure is given by the coCartesian fibration $\O^\otimes\stackrel{\id}\rightarrow \O^\otimes$. Thus the associated Cartesian fibration $\widehat{\O^\otimes}\rightarrow (\O^\otimes)\op$ is simply $(\O^\otimes)\op\stackrel{\id}\rightarrow (\O^\otimes)\op$. Hence, taking opposites again, the opposite $\O$-monoidal structure of $\O$ is given by the coCartesian fibration $\O^\otimes\stackrel{\id}\rightarrow \O^\otimes$, i.e. is $\O$ again. Hence we obtain:
\[
 \Funcolax_\O(\O, \C)=(\Funlax_\O(\O, \C\op))\op\simeq (\alg_\O(\C\op))\op=\coalg_\O(\C). \qedhere
\]
\end{proof}

\begin{rem}\label{rem: strong monoidal=strong comonoidal}
A functor $F:\C\rightarrow \D$ is a strong $\O$-monoidal functor if and only if its opposite $F:\C\op\rightarrow \D\op$ is a strong $\O$-monoidal functor. This follows from \cite[1.3]{dualcocart}. In particular, we recover the well-known result that strong $\O$-monoidal functors are colax $\O$-monoidal.
\end{rem}

In order to define the cyclic bar and cobar construction in the $\infty$-category setting, we use the approach of \cite{tch}. For any symmetric monoidal $\infty$-category $\Co$ (or more generally any $\infty$-operad), we can define its active part $(\Co)_\mathsf{act}$ as the pullback (see \cite[2.2.4.1]{HA})
\[
\begin{tikzcd}
(\Co)_\mathsf{act} \ar{r} \ar{d} & N(\Fin) \ar{d}\\
\Co \ar{r} & N(\Fin_*).
\end{tikzcd}
\]
Here $\Fin\subseteq \Fin_*$ is the subcategory consisting of the active morphisms.
The fiber of $(\Co)_\mathsf{act}$ over a finite set $I$ in $\Fin$ is given by $\C^I$.
An $\ee$-algebra $A$ in $\C$ is a lax monoidal functor  $A^\ot:\assoc^\ot\rightarrow \Co$.
It extends to $A^\ot: (\assoc^\ot)_\mathsf{act} \rightarrow (\Co)_\mathsf{act}$.
The adjoint of the identity functor $\C\rightarrow \C$ is the functor
\[
\ot:(\Co)_\mathsf{act} \rightarrow \C,
\]
defined by $(X_1, \ldots, X_n)\rightarrow X_1\ot \ldots \ot X_n$, as in \cite[III.3.2]{tch}.
From \cite[B.1, B.2]{tch}, we obtain a functor:
\[
N(\Delta\op) \longrightarrow (\assoc^\otimes)_\mathsf{act}.
\]
We therefore obtain a simplicial object $\bar^\C(A)$ in $\C$, i.e. an object in the $\infty$-category $\Fun(N(\Delta\op), \C)$, as the composite of functors:
\[
\begin{tikzcd}
N(\Delta\op)\ar{r} & (\assoc^\otimes)_\mathsf{act} \ar{r}{A^\ot} & (\Co)_\mathsf{act} \ar{r}{\ot} & \C.
\end{tikzcd}
\]
This is making precise the cyclic bar construction in $\C$:
\[
\begin{tikzcd}
 \cdots \ar[shift left=3]{r}\ar[shift right=3]{r}\ar[shift left]{r}\ar[shift right]{r} &A\ot A \ot A \ar{r}\ar[shift left=2]{r}\ar[shift right=2]{r} &A\ot A\ar[shift left]{r}\ar[shift right]{r} & A.
\end{tikzcd}
\]

\begin{rem}\label{rem: our cylic bar is the usual one}
If $\C$ is the nerve of a symmetric monoidal ordinary category $\M$ and $A$ is a (strictly) associative algebra in $\M$, then $\bar^\C(A)$ corresponds to the usual cyclic bar complex of $A$ in $\M$.
\end{rem}

\begin{defi}
Let $\C$ be a symmetric monoidal $\infty$-category that admits geometric realizations.
For $A$ an $\ee$-algebra in $\C$, we define its \emph{topological Hoschchild homology} $\thh(A)$ to be the geometric realization in $\C$ of the simplicial object $\bar^\C(A)$.
\end{defi}

\begin{rem}
By \cite[B.5]{tch}, our definition of $\thh(A)$ is equivalent to the one of \cite[III.2.3]{tch} where we ignore the cyclotomic structures.
\end{rem}

Since our description of the cyclic bar construction is completely natural in the $\ee$-algebra $A$, it defines a functor:
\[
\bar^\C(-):\alg_\ee(\C)\longrightarrow \Fun(N(\Delta\op), \C),
\]
for a symmetric monoidal $\infty$-category $\C$. In particular, if we apply this to the opposite symmetric monoidal $\infty$-category $\C\op$, we obtain the functor:
\[
\alg_\ee(\C\op)\longrightarrow \Fun(N(\Delta\op), \C\op).
\]
Taking opposite on each side of this functor, we obtain:
\[
\Big(\alg_\ee(\C\op)\Big)\op \longrightarrow \Big(\Fun(N(\Delta\op), \C\op)\Big)\op=\Fun(N(\Delta),(\C\op)\op),
\]
i.e. we constructed a functor
\[
\cobar_\C(-):\coalg_\ee(\C)\longrightarrow \Fun(N(\Delta), \C).
\]
This defines a \emph{cobar} cyclic complex in $\C$:
\[
\begin{tikzcd}
 \cdots & \ar[shift left=3]{l}\ar[shift right=3]{l}\ar[shift left]{l}\ar[shift right]{l} C\ot C\ot C & \ar{l}\ar[shift left=2]{l}\ar[shift right=2]{l} C\ot C& \ar[shift left]{l}\ar[shift right]{l}  C.
\end{tikzcd}
\]

\begin{defi}\label{defi: cothh}
Let $\C$ be a symmetric monoidal $\infty$-category that admits totalizations.
Let $C$ be an $\ee$-coalgebra in $\C$.
Its \emph{topological coHochschild homology} $\cothh(C)$ is defined to be the totalization in $\C$ of the cosimplicial object $\cobar_\C(C)$.
\end{defi}

\begin{notation}
We shall sometimes write $\cothh^\C(C)$ to emphasize the symmetric monoidal $\infty$-category $\C$ that is being considered. 
If $R$ is an $\ei$-algebra in $\C$, we may write $\cothh^{\Mod_R(\C)}(C)$ simply by $\cothh^R(C)$, where $C$ is an $\ee$-coalgebra in $\Mod_R(\C)$. 
For $\C= \sp$ we follow the notation of the algebraic case and denote by $\cothh^R_*(C)$ the homotopy groups $\pi_*( \cothh^R(C))$, for any $R$-coalgebra $C$. {Furthermore if  $R=\sph$, we omit $\sph$ and write $\cothh(C)$ instead of $\cothh^\sph(C)$.}
\end{notation}

Let $\M$ be a \emph{symmetric monoidal model category} (see \cite[4.2.6]{hovey}), and denote by $\W$ its class of weak equivalences. Let $\M_c$ denote the full subcategory spanned by the cofibrant objects. Let us denote $\MM$ its underlying symmetric monoidal $\infty$-category as defined in \cite[4.1.7.6]{HA}. This is also sometimes called the \emph{(symmetric monoidal) Dwyer-Kan localization} of the (symmetric monoidal) model category $\M$.

In \cite[2.3]{toolscothh}, the authors defined the topological coHochschild homology $\cothh^\M(C)$ of a (strictly) coassociative counital coalgebra $C$ in $\M$ to be the homotopy limit in $\M$ of the cosimplicial object given by the cyclic cobar complex. We show that our definition agrees with the one in model categories in Proposition \ref{Prop: def of cothh agrees} below.

\begin{lem}\label{lem: strong monoidal preserves cyclic bar}
Let $F:\C\rightarrow \D$ be a strong symmetric monoidal functor between symmetric monoidal $\infty$-categories. Then: {\color{blue} }
\[
F\circ\bar^\C(-) \simeq \bar^\D(-)\circ F.
\]
Dually, we also obtain:
\[
F\circ\cobar_\C(-) \simeq \cobar_\D(-)\circ F.
\]
\end{lem}

\begin{proof}
By \cite[4.4]{liam}, we obtain the commutative diagram of $\infty$-categories:
\[
\begin{tikzcd}
(\C^\otimes)_{\mathsf{act}} \ar{d}[swap]{F_\mathsf{act}} \ar{r}{\otimes} & \C \ar{d}{F}\\
(\D^\otimes)_\mathsf{act} \ar{r}{\otimes} & \D.
\end{tikzcd}
\]
Let $A$ be an $\ee$-algebra in $\C$. By the above result, we obtain that the composite
\[
\begin{tikzcd}
(\assoc^\otimes)_\mathsf{act} \ar{r}{A^\otimes} & (\C^\otimes)_{\mathsf{act}} \ar{r}{\otimes} & \C \ar{r}{F}
 & \D\end{tikzcd}
\]
is equivalent to the composite
\[
\begin{tikzcd}
(\assoc^\otimes)_\mathsf{act} \ar{r}{F(A)^\otimes} & (\D^\otimes)_\mathsf{act} \ar{r}{\otimes} & \D. 
\end{tikzcd}
\]
Therefore, we obtain $F\left(\bar^\C(A)\right) \simeq \bar^\D(F(A))$. We obtain the dual statement for cobar by applying our above argument to the strong symmetric monoidal functor $F\op:\C\op\rightarrow \D\op$. 
\end{proof}

\begin{prop}\label{Prop: def of cothh agrees}
Let $\M$ be a combinatorial symmetric monoidal model category.
Let $C$ be a coassociative coalgebra in $\M$ that is cofibrant as an object in $\M$.
Then $\cothh^\M(C)$ is weakly equivalent to $\cothh^\MM(C)$.
\end{prop}

\begin{proof}
In \cite[3.1]{perouxDKloc}, there is a functor of $\infty$-categories:
\[
\iota:N(\mathsf{coAlg}(\M_c))\big[\Wwo\big]\longrightarrow \coalg_\ee\left(\MM\right),
\]
which is the identity on objects in $\MM$. Here $\Ww$ denotes the weak equivalences in $\M$ that are also maps in $\coalg(\M_c)$. 
The functor $\iota$ allows one to consider the (strictly) coassociative counital coalgebra $C$ in $\M$ as an $\ee$-coalgebra $\iota(C)$ in the Dwyer-Kan localization of $\M$. 
Since the localization $N(\M_c)\rightarrow \MM$ is a strong symmetric monoidal functor, then by Remark \ref{rem: our cylic bar is the usual one}, \cite[1.3.4.25]{HA}, and Lemma \ref{lem: strong monoidal preserves cyclic bar}, the two definitions of the cyclic cobar complexes agree. The desired result follows by \cite[1.3.4.23]{HA}, as the homotopy limit of the cyclic cobar construction of $C$ in $\M$ is equivalent to the totalization of the cyclic cobar construction of $\iota(C)$ in the Dwyer-Kan localization of $\M$.
\end{proof}

\begin{rem}
In the proof of Proposition \ref{Prop: def of cothh agrees}, we are not claiming that the $\infty$-category $N(\mathsf{coAlg}(\M_c))[\Wwo]$ is the Dwyer-Kan localization of a model structure on $\coalg(\M)$ as such a model structure may not exist, see Remark \ref{rem: cothh difficult in model categories} below. Here we are instead referring to the Dwyer-Kan localization of the $\infty$-category $N(\mathsf{coAlg}(\M_c))$ with respect to the class of edges induced by the weak equivalences $\Ww$, as in \cite[1.3.4.1]{HA}. 
\end{rem}

\begin{rem}\label{rem: cothh difficult in model categories}
The model category approach faces several challenges. First of all, given a combinatorial symmetric monoidal model category $\M$, there are no conditions in general to lift {the model structure on $\M$ to} a model structure on the category of (strictly) coassociative counital coalgebras $\coalg(\M)$. In the cases where a model structure exists, the model structure of $\M$ is generally replaced by a Quillen equivalent one that is not a monoidal model category, see \cite{left2, left3}. Second, given a model structure in $\coalg(\M)$ induced by one on $\M$, there is no reason to expect that strictly coassociative coalgebras in $\M$ are equivalent to $\ee$-coalgebras in the Dwyer-Kan localization of $\M$, unlike the case for algebras (as in \cite[4.1.8.4]{HA}). As a matter of fact, there is indication that these do \emph{not} correspond. In particular, for $\M$ any symmetric monoidal model category of spectra, such as symmetric spectra or orthogonal spectra, it is shown that strictly coassociative coalgebras do \emph{not} correspond to $\ee$-coalgebras in the $\infty$-category of spectra $\sp$, see \cite{perouxshipley} and \cite[5.6]{perouxDKloc}. 
\end{rem}

\section{Duality Between Algebras and Coalgebras}\label{sec general duality between algebras and coalgebras}
From a coalgebra, one can always obtain an associated algebra by taking the dual. The converse is in general not true. We make this fact precise in this section. 
We also show here equivalences between subclasses of algebras and coalgebras that we introduce. 

\subsection{The Linear Dual Functor} {We introduce  the notion of dualization for closed symmetric monoidal $\infty$-categories.} We recall the well-known result that the dual of a coalgebra is always an algebra.

\begin{defi} \label{def duality functor}
Let $\C$ be a closed symmetric monoidal $\infty$-category. Denote \[[-,-]:\C\op\times \C\rightarrow \C,\] its internal hom defined by the universal property $\C(X\ot Y, Z) \simeq \C(X, [Y, Z]).$ It is lax symmetric monoidal, see \cite[A.5.3]{hinrect} or \cite[I.3]{runelax}.
Let $\I$ be the unit of the symmetric monoidal structure of $\C$. 
For any object $X$ in $\C$, define $X^\vee$ to be $[X, \I]$. This defines a lax symmetric monoidal functor: \[(-)^\vee \co \C\op\rightarrow \C,\] called the \emph{linear dual functor}. We denote its opposite $((-)^\vee)^\op \co \C \to \C\op$ again by:
\[(-)^\vee \co \C \to \C\op,\]
and also refer to it as the linear dual functor. 
\end{defi}

\begin{ex}
If $\C$ is the nerve of the category of vector spaces over a field $k$, then the linear dual functor $(-)^\vee$ is precisely the dual of a vector space over $k$.
\end{ex}

\begin{ex}
{If $\C$ is the $\infty$-category of spectra $\sp$, then the linear dual is precisely the Spanier-Whitehead dual.}
\end{ex}

The linear dual is what is called a \emph{self-adjoint (contravariant) functor}: a property that we describe in the following result.

\begin{prop}\label{prop: linear dual is self-adjoint}
Let $\C$ be a closed symmetric monoidal $\infty$-category. Then we obtain an adjunction:
\[
\begin{tikzcd}[column sep= large]
\C \ar[shift left=2]{r}{(-)^\vee}[swap]{\perp} & \C\op. \ar[shift left=2]{l}{(-)^\vee}
\end{tikzcd}
\]
\end{prop}

\begin{proof}
This is a classical result in ordinary categories. It follows from the string of equivalences:
\begin{eqnarray*}
\C\op(X^\vee, Y) & \simeq & \C(Y , X^\vee)\\
& \simeq & \C(Y\ot X, \I) \\
& \simeq & \C(X\ot Y, \I)\\
& \simeq & \C(X, Y^\vee),
\end{eqnarray*}
for every $X$ and $Y$ in $\C$. {Here we used the natural symmetry equivalence $X\ot Y\simeq Y \ot X$.} If we choose $Y=X^\vee$, then the identity morphism $X^\vee\rightarrow X^\vee$ in $\C\op$ defines a natural morphism $X\rightarrow X^{\vee\vee}$ via the equivalence
\[
\C(X, X^{\vee\vee})\simeq \C\op(X^\vee, X^\vee),
\]
defined above. This provides the  unit of the adjunction (see \cite[5.2.2.8]{htt}).
\end{proof}

The right adjoint functor $(-)^\vee:\C\op\rightarrow \C$, being lax symmetric monoidal, lifts to the category of algebras
\[
(-)^\vee:\coalg_\O(\C)\op\simeq{\alg_\O(\C\op)} \longrightarrow \alg_\O(\C),
\]
for any $\infty$-operad $\O$. This functor provides a description of the well-known fact that the dual of a coalgebra is always an algebra.
Since limits in algebras are computed in their underlying category, the linear dual $(-)^\vee:\coalg_\O(\C)\op\rightarrow \alg_\O(\C)$ remains a limit-preserving functor. 

However, as the linear dual is not strong monoidal in general, it does not lift to a functor $(-)^\vee:\alg_\O(\C)\rightarrow \coalg_\O(C)\op$.
In other words, if $A$ is an $\O$-algebra in $\C$, then we cannot claim $A^\vee$ is an $\O$-coalgebra in $\C$ in general. We are interested in the cases where an $\O$-algebra $A$ induces an $\O$-coalgebra structure on its linear dual via a comultiplication induced by:
\[
\begin{tikzcd}
A^\vee\ar{r} & (A\otimes A)^\vee\\
& A^\vee \otimes A^\vee, \ar{u}
\end{tikzcd}
\]
where the vertical map is an equivalence. 
We are also interesting in knowing conditions for which the linear dual functor $(-)^\vee:\coalg_\O(\C)\op\rightarrow \alg_\O(\C)$ is an equivalence of $\infty$-categories. 

\subsection{Proper Algebras and Coalgebras} It is a well-known result that algebras and coalgebras over finite dimensional vector spaces are anti-equivalent. We remind here the result of \cite{lurie2} that shows this in greater generality.

\begin{defi}[{\cite[4.6.1]{HA}}]
An object $X$ is said to be \emph{dualizable} in a symmetric monoidal $\infty$-category $\C$ if there exists another object $Y$ in $\C$, together with maps:
\[
e:Y\ot X \rightarrow \I, \qquad c:\I\rightarrow X\ot Y,
\]
for which the composite maps:
\[
\begin{tikzcd}
X\simeq \I\ot X \ar{r}{c\ot \id} & X\ot Y \ot X \ar{r}{\id\ot e} & X\ot \I \simeq X,
\end{tikzcd}
\]
\[
\begin{tikzcd}
Y \simeq Y\ot \I \ar{r}{\id\ot c} & Y\ot X \ot Y \ar{r}{e\ot \id} & \I\ot Y\simeq Y,
\end{tikzcd}
\]
are homotopic to the identity. In other words, an object $X$ is dualizable if and only if it is dualizable in the homotopy category $h\C$, see \cite[4.6.1.6]{HA}. Let $\C_\dual$ denote the full subcategory of $\C$ spanned by the dualizable objects.  By \cite[3.2.4]{lurie2}, the tensor product of $\C$ is closed in $\C_\dual$ and thus $\C_\dual$ is also symmetric monoidal. 
\end{defi}

\begin{ex}[{\cite[3.2.1]{lurie2}}]\label{ex: dualizable=fin dim in vector spaces}
If $\C$ is the nerve of the  category of vector spaces over a field $k$, then the dualizable objects are precisely the finite dimensional vector spaces over $k$.
\end{ex}

\begin{ex}
The dualizable objects in the $\infty$-category of spectra $\sp$ are precisely the finite spectra.
\end{ex}
\begin{ex}[{\cite[3.2.3]{lurie2}}]\label{ex: dualizable in R-modules}
Let $R$ be a commutative ring spectrum. Let $\C$ be the $\infty$-category of $R$-modules $\Mod_R(\sp)$. Then the dualizable objects in $\C$ are precisely the compact objects in $\C$, i.e. the perfect $R$-modules. {An $R$-module is said to be perfect if it lies in the smallest subcategory of $\textup{Ho}(\Mod_R(\sp))$ that contains $R$ and is closed under triangles and retracts.}
\end{ex}

\begin{ex}\label{ex dualizable in hk modules}
Let $k$ be a field. An $Hk$-module $X$ in $\sp$ is dualizable in $\Mod_{Hk}(\sp)$ if and only if $\pi_*X$ is finitely generated as a graded $k$-module.
\end{ex}

\begin{lem}\label{lem: dualizable in closed monoidal cat}
Let $\C$ be a closed symmetric monoidal $\infty$-category. If $X$ is dualizable in $\C$, then its dual is equivalent to the linear dual $X^\vee$.
\end{lem}

\begin{proof}
Let $X$ be a dualizable object in $\C$, with dual $Y$. Then $Y\otimes -:\C\rightarrow \C$ is a right adjoint of $X\otimes-:\C\rightarrow \C$. The unit and counit of the adjunction are induced by the maps $e:Y\otimes X\rightarrow \I$ and $c:\I\rightarrow X\otimes Y$. The universal property of tensor product with its internal hom gives the equivalence of functors on $\C$:
\[
[X, -]\simeq Y\otimes -,
\]
as they are both the right adjoint of the functor $X\ot-:\C\rightarrow \C$.
Then in particular we get the equivalence $X^\vee\simeq Y$. 
\end{proof}

\begin{prop}[{\cite[3.2.4]{lurie2}}]\label{prop: antiequivalence between dualizable objects}
Let $\C$ be a closed symmetric monoidal $\infty$-category. Then the linear dual determines an equivalence of symmetric monoidal $\infty$-categories:
\[
(-)^\vee: (\C_\dual)\op \stackrel{\simeq}\longrightarrow \C_\dual,
\]
with the inverse equivalence given by itself.
\end{prop}

\begin{defi}[{\cite[4.6.4.2]{HA}}]
Let $\C$ be a symmetric monoidal $\infty$-category. Let $\O$ be an $\infty$-operad. An $\O$-algebra $A$ in $\C$ is said to be \emph{proper} if $A$ is dualizable as an object in $\C$. Similarly, we say that an $\O$-coalgebra $C$ in $\C$ is \emph{proper} if it is dualizable as an object in $\C$.
\end{defi}

There is an anti-equivalence between proper algebras and proper coalgebras, over any $\infty$-operad, given by the linear dual functor. This follows directly from Proposition \ref{prop: antiequivalence between dualizable objects}.

\begin{cor}[{\cite[3.2.5]{lurie2}}]\label{cor: antiequivalence proper (co)algebras}
Let $\C$ be a closed symmetric monoidal $\infty$-category. Let $\O$ be an $\infty$-operad. Then the linear dual induces an equivalence: 
\[
(-)^\vee: \coalg_\O(\C_\dual)\op \stackrel{\simeq}\longrightarrow \alg_\O(\C_\dual),
\]
{where the inverse equivalence is given by itself.}
\end{cor}

\subsection{Quasi-Proper Algebras and Coalgebras} We want to generalize the anti-equivalence between proper algebras and coalgebras of Corollary \ref{cor: antiequivalence proper (co)algebras}, by considering a larger class of algebras and coalgebras. We do so in Theorem \ref{thm: anti-equivalence between quasi-proper (co)algebras} below. We introduce the notions of quasi-proper algebras and coalgebras. This allows us to extend the equivalence between THH and coTHH (Theorem \ref{thm cothh duality general}) in the next section, see Theorem \ref{thm most general thh cothh duality}.
Given $A$ an algebra, we are also seeking conditions for its linear dual $A^\vee$ to be a coalgebra.

\begin{cond}\label{cond 1} 
Let $\C$ be a closed symmetric monoidal $\infty$-category.
We say an object $X$ in $\C$ satisfies this condition if the lax symmetric monoidal structure map of the linear dual functor $(-)^\vee:\C\op\rightarrow \C$:
\[{(X^\vee)}^{\ot n} \longrightarrow {(X^{\ot n})}^\vee\]
is an equivalence in $\C$ for all $n \geq 0$. 
\end{cond}

\begin{cond}\label{cond 2}
Let $\C$ be a closed symmetric monoidal $\infty$-category.
We say an object $X$ in $\C$ satisfies this condition if the natural unit map from the self-adjoint property of the linear dual
\[X \longrightarrow X^{\vee\vee}\]
is an equivalence in $\C$.
\end{cond}

\begin{rem}
In Condition \ref{cond 1}, the case $n=0$ is stating that $\I^\vee\simeq \I$ which is always satisfied, where $\I$ is the unit of the closed symmetric monoidal $\infty$-category $\C$. The case $n=1$ is also always satisfied for any $X$ in $\C$.
\end{rem}

We provide examples of spectra and examples of chain complexes that do not satisfy these conditions. These  examples provide a justification for the hypothesis of Theorem \ref{thm cothh duality in chains}.

\begin{ex} \label{ex hz is counterexample to cond 2}
The Eilenberg-Mac Lane spectrum $\hfp$ satisfies Condition \ref{cond 1} but does not satisfy Condition \ref{cond 2} in the $\infty$-category of spectra $\sp$.
It is known that the Spanier–Whitehead dual of $\hfp$ is trivial \cite{lin1976dualityandeilenbergmaclane}. This shows that \[\hfp^{\vee\vee}\simeq 0\]
and that $\hfp$ does not satisfy Condition \ref{cond 2} in $\sp$. Furthermore, we have the equivalences 
\[[\hfp^{\wdg n},\sph]\simeq [\hfp^{\wdg (n-1)},[\hfp,\sph]] \simeq 0\]
and therefore both sides of the map in Condition \ref{cond 1} are trivial. Therefore, $\hfp$ satisfies Condition 1 in $\sp$.
\end{ex}

\begin{ex} \label{ex cond 2 doesnt imply cond 1}
We provide here an example of an object that satisfies Condition \ref{cond 2} but not Condition \ref{cond 1} in the $\infty$-category of $Hk$-modules $\Mod_{Hk}(\sp)$, where $k$ is a field. The homotopy category of $Hk$-modules is equivalent to the category of graded $k$-modules as a symmetric monoidal category. This follows from the K\"unneth and Ext spectral sequences in \cite[\RomanNumeralCaps{4}.4.1]{elmendorf2007rings}.
Let $X$ denote the $Hk$-module corresponding to the chain complex $k[x_1,x_1^{-1}]$ with trivial differentials where $\lv x_1 \rv = 1$. Note that the homology of this chain complex and therefore the homotopy of the corresponding $Hk$-module is $k[x_1,x_1^{-1}]$. Since  dualization in $Hk$-modules results in graded $k$-module dualization at the level of homotopy groups,  we have: 
\[\pi_*(X^\vee) \cong k[x_1,x_1^{-1}]^\vee \cong k[x_1,x_1^{-1}].\]
 By inspection, it is clear that $X$ satisfies Condition \ref{cond 2} as an $Hk$-module. Note that $X \wdg_{Hk} X$ has  countably infinite dimensional homotopy groups at each degree. For instance, the basis for degree zero homotopy group is given by $x_1^i \otimes x_1^{-i}$ over all $i \in \z$. Therefore  $(X \wdg_{Hk} X)^\vee$ has \textit{uncountably} infinite dimensional homotopy at each degree. On the other hand, $X^\vee \wdg_{Hk} X^\vee$ has countably infinite dimensional homotopy at each degree. Therefore the map
\[X^\vee \wdg_{Hk}X^\vee \to (X \wdg_{Hk} X)^\vee \]
 cannot be surjective in homotopy showing that $X$ does not satisfy Condition \ref{cond 1}. This example justifies the (co)connectivity hypothesis in Theorem \ref{thm cothh duality in chains}.
\end{ex}

\begin{ex}\label{ex counter to cond 1}
 Let $X$ be a countably infinite dimensional $k$-vector space where $k$ is a field. Let $HX$ denote the corresponding $Hk$-module. Neil Strickland explains\footnote{\url{https://mathoverflow.net/questions/56255/duals-and-tensor-products}} that $HX$ does not satisfy Condition \ref{cond 1} in $Hk$-modules because the map mentioned in Condition \ref{cond 1} is not surjective in homotopy for $HX$. To see this, let $\{x_i\}_{i \in \z}$ be a basis for $X$. We consider the map in Condition \ref{cond 1} for the case $n=2$. At the level of homotopy, this map
 \[\varphi\co X^\vee \ot_k X^\vee \to (X \ot_k X)^\vee\]
 is given by: 
 \[\varphi(f \ot_k g) (x \ot_k y) = f(x) g(y).\]
 The $k$-linear map $h \co X \ot_k X \to k$ given by $h(x_i\ot_k x_j) = 0$ if $i \neq j$ and $h(x_i\ot_k x_i) = 1$ for all $i$ is not in the image of $\varphi$. This is due to the fact that the matrix $(h(x_i \ot_k x_j)_{i,j})$ has infinite rank where the matrix  $(\varphi(\Sigma_{\ell=1}^rf_\ell \otimes g_\ell) (x_i \ot_k x_j)_{i,j})$ has rank at most $r$. Therefore, $\varphi$ is not surjective and $X$ does not satisfy Condition \ref{cond 1} in $Hk$-modules. This example justifies the finiteness hypothesis in Theorem \ref{thm cothh duality in chains}.
 Furthermore, dualization increases cardinality of an infinite dimensional vector space. In particular, $\pi_0((HX)^{\vee \vee})= X^{\vee\vee}$ has a larger cardinality than $\pi_0HX=X$ and therefore $X$ does not satisfy Condition \ref{cond 2} either. 
 
\end{ex}

We now record some results that capture if the linear dual preserves the conditions above. 

\begin{lem}\label{lem: cond 2 preserved under dual}
Let $\C$ be a closed symmetric monoidal $\infty$-category. Let $X$ be an object in $\C$.
If $X$ satisfies Condition \ref{cond 2}, then $X^\vee$ satisfies Condition \ref{cond 2}. 
\end{lem}

\begin{proof}
We have the equivalence:
\[
X^\vee \rightarrow (X^\vee)^{\vee\vee}\simeq (X^{\vee\vee})^\vee \simeq X^\vee,
\]
using Condition \ref{cond 2} on $X$.
\end{proof}

\begin{lem}\label{lem: cond 3}
Let $\C$ be a closed symmetric monoidal $\infty$-category. Suppose $X$ satisfies Conditions \ref{cond 1} and \ref{cond 2}. Then $X^\vee$ satisfies Condition \ref{cond 1} if and only if $X^{\ot n}$ satisfies Condition \ref{cond 2} for all $n\geq 0$.
\end{lem}

\begin{proof}
We have the following commutative diagram in $\C$: 
\[
\begin{tikzcd}
X^{\ot n} \ar{r} \ar{d}[swap]{\simeq} & {(X^{\ot n})}^{\vee\vee} \ar{d}{\simeq}\\
{(X^{\vee\vee})}^{\ot n} \ar{r} & {\left( (X^\vee)^{\ot n}\right)}^\vee.
\end{tikzcd}
\]
The left {vertical} map is an equivalence as $X$ satisfies Condition \ref{cond 2}. The right {vertical} map is an equivalence as $X$ satisfies Condition \ref{cond 1}. The result follows.
\end{proof}

\begin{defi}\label{def: quasi-dualizable}
Let $\C$ be a closed symmetric monoidal $\infty$-category. An object $X$ in $\C$ is said to be \emph{quasi-dualizable} if $X$ satisfies both Conditions \ref{cond 1} and \ref{cond 2}. It is called \emph{\XXX-dualizable}  if $X$ satisfies Condition \ref{cond 1}.
\end{defi}

\begin{cor}\label{cor dualizable satisfy cond 1 and 2}
Let $\C$ be a closed symmetric monoidal $\infty$-category. If an object in $\C$ is dualizable then it is quasi-dualizable in $\C$.
\end{cor}

\begin{proof}
Apply Lemma \ref{lem: dualizable in closed monoidal cat} and Proposition \ref{prop: antiequivalence between dualizable objects}. 
\end{proof}

\begin{rem}\label{rem: different notions of dualizable}
Corollary \ref{cor dualizable satisfy cond 1 and 2} shows that we have the implications:
\[
\text{dualizable} \Rightarrow \text{quasi-dualizable} \Rightarrow \text{\XXX-dualizable}.
\]
Example \ref{ex hz is counterexample to cond 2} provides an instance of an object that is \XXX-dualizable but not quasi-dualizable; therefore, weak quasi-dualizability is a strictly weaker condition than quasi-dualizability. For a Noetherian ring $A$ with finite global dimension, we show in Section \ref{sec proof of cothh thh duality in chains} that a (co)connective finite type $HA$-module is quasi-dualizable. If such an $HA$-module has unbounded homotopy, then it is not dualizable due to Proposition \ref{prop dualizable ha modules}. This shows that quasi-dualizability is a strictly weaker condition than dualizability. For instance, $\hfp \wdg \hfp $ is quasi-dualizable but \emph{not} dualizable in $\hfp$-modules. 
\end{rem}

\begin{rem}\label{rem dualizable in k modules are finite dimensional}
The distinction made in the previous remark is not seen in the classical case.
Indeed, let $\C$ be the nerve the category of vector spaces over a field $k$.
We saw in Example \ref{ex: dualizable=fin dim in vector spaces} that a vector space is dualizable if and only if it is finite dimensional. In fact, the discussion in Example \ref{ex counter to cond 1} shows that infinite dimensional vector spaces cannot satisfy Condition \ref{cond 1} nor Condition \ref{cond 2}. Therefore, given a vector space $V$ over $k$, the following are equivalent:
\begin{enumerate}
    \item $V$ is finite dimensional,
    \item $V$ is dualizable,
    \item $V$ is quasi-dualizable,
    \item $V$ is weak quasi-dualizable, i.e. satisfies Condition \ref{cond 1},
    \item $V$ satisfies Condition \ref{cond 2}.
\end{enumerate}
\end{rem}

\begin{defi}\label{defi: weak-quasi proper algebra}
Let $\C$ be a closed symmetric monoidal $\infty$-category. Let $\O$ be an $\infty$-operad.  We say an $\O$-algebra $A$ in $\C$ is \emph{\XXX-proper} {if} $A$ is \XXX-dualizable as an object in $\C$. 
We denote by $(\alg_\O(\C))_\xdual$ the full subcategory in $\alg_\O(\C)$ spanned by the \XXX-proper algebras.
\end{defi}

We first show that a \XXX-dualizable algebra $A$ has the property that its linear dual $A^\vee$ has a natural induced coalgebra structure.

\begin{prop}\label{prop: cond 1 implies dual of algebra is a coalgebra}
Let $\C$ be a closed symmetric monoidal $\infty$-category.
Let $\O$ be an $\infty$-operad. Let $A$ be a \XXX-proper $\O$-algebra in $\C$, then $A^\vee$ is an $\O$-coalgebra in $\C$. More generally, the linear dual functor induces a functor:
\[
(-)^\vee: (\alg_\O(\C))_\xdual \longrightarrow (\coalg_\O(\C))\op.
\]
\end{prop}

We prove the above proposition after showing the following lemma.

\begin{lem}\label{lem: lax monoidal becomes strong conditions}
Let $\C$ and $\D$ be symmetric monoidal $\infty$-categories. Let $\I_\C$ and $\I_\D$ be the unit objects of the monoidal structures of $\C$ and $\D$ respectively. Let $F:\C\rightarrow \D$ be a lax symmetric monoidal functor. If for all objects $X$ and $Y$ in $\C$, the natural maps $F(X)\ot F(Y)\rightarrow F(X\ot Y)$ are equivalences, and the natural map $\I_\D\rightarrow F(\I_\C)$ is an equivalence, then $F$ is a strong symmetric monoidal functor.
\end{lem}

\begin{proof}
Let $p:\C^\ot\rightarrow N(\Fin_*)$ and $q:\D^\ot\rightarrow N(\Fin_*)$ be the coCartesian fibrations that endow $\C$ and $\D$ of a symmetric monoidal structure. We need to show that $F$ sends $p$-coCartesian lifts to $q$-coCartesian lifts.

We first recall how we obtain the natural maps $F(X)\ot F(Y)\rightarrow F(X\ot Y)$.
Let $m:\langle 2\rangle\rightarrow \langle 1\rangle$ be the map in $\Fin_*$ that sends both $1$ and $2$ to $1$. Let $X$ and $Y$ be objects in $\C$. On the one hand, we have the $p$-coCartesian lift of $m$ with respect to $(X,Y)$ given by the edge $(X,Y)\rightarrow X\otimes Y$ in $\C^\ot$ that represents the tensor product:
\[
\begin{tikzcd}
\C\times \C & \C_{\langle 2 \rangle}^\otimes \ar{l}[swap]{\simeq} \ar{r}{m_!} & \C.
\end{tikzcd}
\]
Similarly, we obtain the $q$-coCartesian lift $(F(X), F(Y))\rightarrow F(X)\otimes F(Y)$ in $\D^\ot$. On the other hand, since $m$ is not an inert morphism in $\Fin_*$, then the induced map $F(X,Y)\rightarrow F(X\otimes Y)$ is not a priori a $q$-coCartesian lift. However, since the Segal condition $\C^\ot_{\langle n \rangle}\simeq \C^{\times n}$ is induced by the inert morphisms $\langle n \rangle \rightarrow \langle 1 \rangle$, and  $F$ is lax symmetric monoidal, we obtain an equivalence $F(X, Y)\stackrel{\simeq}\rightarrow (F(X), F(Y))$. 
By the universal property of coCartesian lifts, we obtain a unique map up to contractible choice $F(X)\otimes F(Y)\rightarrow F(X\otimes Y)$ such that we have the commutative diagram:
\[
\begin{tikzcd}
(F(X), F(Y)) \ar{r} \ar{dr} & F(X)\ot F(Y) \ar[dashed]{d}\\
& F(X \ot Y).
\end{tikzcd}
\]
If we assume that the dashed map is an equivalence, then by unicity of the coCartesian lifts, we see that $F$ sends $p$-coCartesian lifts of $m$ to $q$-coCartesian lifts.

Similarly, the unique map $\iota:\langle 0\rangle \rightarrow \langle 1 \rangle $ in $\Fin_*$ induces the map $\I_\D\rightarrow F(\I_\C)$, and if we require it to be an equivalence, $F$ sends $p$-coCartesian lifts of $\iota$ to a $q$-coCartesian lift.

Notice that  morphisms in $\Fin_*$ are generated under wedge sum, composition, inert morphisms and the maps $\iota$ and $m$. If $\phi:\langle n\rangle \rightarrow \langle k \rangle $ and $\phi':\langle n' \rangle \rightarrow \langle k' \rangle $ are maps in $\Fin_*$, then by the Segal condition and the universality of coCartesian lifts, we obtain an equivalence: 
\[
\begin{tikzcd}
\C_{\langle n + n'\rangle} \ar{r}{(\phi \vee \phi')_!} \ar{d}[swap]{\simeq} & \C_{\langle k + k' \rangle}\ar{d}{\simeq}\\
\C_{\langle n \rangle}\times \C_{\langle n' \rangle} \ar{r}{\phi_! \times \phi'_!} & \C_{\langle k \rangle}\times \C_{\langle k' \rangle}.
\end{tikzcd}
\]
Therefore $F$ sends $p$-coCartesian lifts to $q$-coCartesian lifts. 
\end{proof}

\begin{proof}[Proof of Proposition \ref{prop: cond 1 implies dual of algebra is a coalgebra}]
For any object $X$ in $\C$, let us denote $\C_X$ the full subcategory of $\C$ spanned by objects equivalent to $X^{\ot n}$ for any $n\geq 0$. By \cite[2.2.1.2]{HA}, since $\C_X$ is closed under the tensor product and contains the unit,  the $\infty$-category $\C_X$ inherits a symmetric monoidal structure from $\C$.

Suppose $X$ is weak quasi-dualizable. Then the composition
\[
\begin{tikzcd}
\C_X\ar[hook]{r} & \C\ar{r}{(-)^\vee} & \C\op
\end{tikzcd}
\]
is strong symmetric monoidal, by applying the dual of Lemma \ref{lem: lax monoidal becomes strong conditions} to the colax symmetric monoidal functor $(-)^\vee:\C\rightarrow \C\op$. 

If $X$ is an $\O$-algebra in $\C$, then $X$ is also an $\O$-algebra in $\C_X$ as the inclusion $\C_X\subseteq \C$ is strong symmetric monoidal. 

Therefore, if we let $A$ be a weak quasi-proper $\O$-algebra in $\C$, we apply the discussion above to $\C_A$ and we obtain that $A^\vee$ is an $\O$-coalgebra in $\C$. The coalgebra structure is completely natural from the choice of algebra structure on $A$ and thus we obtain the desired functor.
\end{proof}

Now that we have obtained conditions that ensure the linear dual of an algebra is a coalgebra, we restrict further the conditions to show that the assignment is an equivalence. 

\begin{defi}\label{def quasi proper}
Let $\C$ be a closed symmetric monoidal $\infty$-category. Let $\O$ be an $\infty$-operad. \begin{enumerate}
    \item We say an $\O$-algebra in $\C$ is \emph{quasi-proper} if $A$ is quasi-dualizable as an object in $\C$. We denote by $(\alg_\O(\C))_\qdual$ the full subcategory in $\alg_\O(\C)$ spanned by the quasi-proper algebras.
    
    \item We say an $\O$-coalgebra $C$ is \emph{quasi-proper} if $C^\vee$ satisfies Condition \ref{cond 1} as an object in $\C$, and $C$ satisfies Condition \ref{cond 2} as an object in $\C$. We denote by $(\coalg_\O(\C))_\qdual$ the full subcategory in $\coalg_\O(\C)$ spanned by the quasi-proper coalgebras.
\end{enumerate}
\end{defi}

\begin{rem} 
From Remark \ref{rem: different notions of dualizable}, we obtain the following sequence of subcategories in $\alg_\O(\C)$:
\[\alg_\O(\C_\dual)\subseteq (\alg_\O(\C))_\qdual \subseteq (\alg_\O(\C))_\xdual.\]
\end{rem}

\begin{rem}
We do not define the notion of \XXX-proper coalgebra as the linear dual of a coalgebra is always an algebra since the linear dual functor $(-)^\vee:\C\rightarrow \C\op$ is colax symmetric monoidal. We have defined the notion of quasi-proper coalgebra so that it is exactly the {essential} image of the linear dual functor defined in Proposition \ref{prop: cond 1 implies dual of algebra is a coalgebra} restricted to quasi-proper algebras, as we show in Theorem \ref{thm: anti-equivalence between quasi-proper (co)algebras} below.
\end{rem}

\begin{prop} \label{prop cond 1 2 and cond 2 on c ot c imply quasi-proper}
Let $\C$ be a closed symmetric monoidal $\infty$-category. Let $\O$ be an $\infty$-operad. Let $C$ be an $\O$-coalgebra in $\C$.
If $C$ satisfies Conditions \ref{cond 1} and \ref{cond 2} as an object in $\C$, and $C^{\otimes n}$ satisfies Condition \ref{cond 2} as an object in $\C$ for all $n\geq 0$, then $C$ is a quasi-proper coalgebra in $\C$. 
\end{prop}

\begin{proof}
We only need to check that $C^\vee$ satisfies Condition \ref{cond 1}, but this follows from Lemma \ref{lem: cond 3}.
\end{proof}

\begin{thm}\label{thm: anti-equivalence between quasi-proper (co)algebras}
Let $\C$ be a closed symmetric monoidal $\infty$-category. Let $\O$ be an $\infty$-operad. Then the linear dual functor induces an equivalence of $\infty$-categories
\[
(-)^\vee: ((\coalg_\O(\C))_\qdual)\op \stackrel{\simeq}\longrightarrow (\alg_\O(\C))_\qdual
\]
{where the inverse equivalence is also given by the linear dual functor $(-)^\vee$.}
\end{thm}

\begin{proof}
We first begin to show that the linear dual functor
\[
(-)^\vee :\coalg_\O(\C)\op \rightarrow \alg_\O(\C),
\]
restricts and corestricts to the  functor:
\[
(-)^\vee: ((\coalg_\O(\C))_\qdual)\op \longrightarrow (\alg_\O(\C))_\qdual.
\]
Let $C$ be a quasi-proper $\O$-coalgebra in $\C$. Then by assumption, its linear dual $C^\vee$ satisfies Condition \ref{cond 1} as an object {in} $\C$. Moreover, it also satisfies Condition \ref{cond 2} by Lemma \ref{lem: cond 2 preserved under dual}. Thus $C^\vee$ is a quasi-proper $\O$-algebra in $\C$. 

We now provide its inverse. In Proposition \ref{prop: cond 1 implies dual of algebra is a coalgebra}, we have shown that the linear dual functor lifts to a functor
\[
(-)^\vee: (\alg_\O(\C))_\xdual \longrightarrow (\coalg_\O(\C))\op.
\]
We now show it restricts and corestricts to:
\[
(-)^\vee: (\alg_\O(\C))_\qdual \longrightarrow ((\coalg_\O(\C))_\qdual)\op.
\]
We need to verify that for $A$ a quasi-proper $\O$-algebra in $\C$, then $A^\vee$ is a quasi-proper $\O$-coalgebra in $\C$. Since $A$ satisfies Condition \ref{cond 2} as an object in $\C$, then so does $A^\vee$ by Lemma \ref{lem: cond 2 preserved under dual}. Since $A$ satisfies Condition \ref{cond 1} and that $A^{\vee\vee}\simeq A$, then $A^{\vee\vee}$ satisfies Condition \ref{cond 1} as an object of $\C$. Thus $A^\vee$ is indeed a quasi-proper $\O$-coalgebra.

We now argue that given a quasi-proper $\O$-algebra $A$, the unit map $A\rightarrow A^{\vee\vee}$ in $\C$ of the adjunction of Proposition \ref{prop: linear dual is self-adjoint} is a map in $\alg_\O(\C)$.
If we let $X$ be a quasi-dualizable object in $\C$, then as in the proof of Proposition \ref{prop: cond 1 implies dual of algebra is a coalgebra}, the linear dual functor restricts to a strong monoidal functor $\C_X\rightarrow \C\op$. Notice that its image is precisely $(\C_{X^\vee})\op$. 
Restricting the linear dual functor $\C\op\rightarrow \C$ to $(\C_{X^\vee})\op$, we obtain a lax symmetric monoidal functor $(\C_{X^\vee})\op\rightarrow (\C_X)$ as we have the equivalence of symmetric monoidal $\infty$-categories $\C_X\simeq \C_{X^{\vee\vee}}$ by \cite[2.1.3.8]{HA} as $X$ satisfies Condition \ref{cond 2}.
Therefore, the adjunction of Proposition \ref{prop: linear dual is self-adjoint} restricts and corestricts to the full subcategories:
\[
\begin{tikzcd}[column sep= large]
\C_X \ar[shift left=2]{r}{(-)^\vee}[swap]{\perp} & (\C_{X^\vee})\op. \ar[shift left=2]{l}{(-)^\vee}
\end{tikzcd}
\]
The left adjoint is strong symmetric monoidal and thus lifts to an adjunction:
\[
\begin{tikzcd}[column sep= large]
\alg_\O(\C_X) \ar[shift left=2]{r}{(-)^\vee}[swap]{\perp} & \alg_\O\left((\C_{X^\vee})\op\right). \ar[shift left=2]{l}{(-)^\vee}
\end{tikzcd}
\]
If we apply the discussion above to a quasi-proper algebra $X=A$, we obtain in particular that the unit of the above adjunction $A\rightarrow A^{\vee\vee}$ is a map of $\O$-algebras in $\C$. We can argue similarly to show that $C\rightarrow C^{\vee\vee}$ is a map of $\O$-coalgebras for $C$ a quasi-proper $\O$-coalgebra.

We can therefore lift the adjunction of Proposition \ref{prop: linear dual is self-adjoint} to an adjunction:
\[
\begin{tikzcd}[column sep= large]
(\alg_\O(\C))_\qdual \ar[shift left=2]{r}{(-)^\vee}[swap]{\perp} & ((\coalg_\O(\C))_\qdual)\op. \ar[shift left=2]{l}{(-)^\vee}
\end{tikzcd}
\]
By \cite[3.2.2.6]{HA}, if $A$ is a quasi-proper $\O$-algebra in $\C$, then unit map of the adjunction $A\stackrel{\simeq}\longrightarrow A^{\vee\vee}$ is an equivalence of $\O$-algebras in $\C$. Similarly if $C$ is a quasi-proper $\O$-coalgebra in $\C$, then the map $C\stackrel{\simeq}\longrightarrow C^{\vee\vee}$ is an equivalence of $\O$-coalgebras in $\C$.
Therefore the unit and counit of the adjunction are equivalences. Thus the adjunction is an equivalence of $\infty$-categories.
\end{proof}

\begin{rem}
In \cite[3.9]{coalgenr}, it was shown that for $\C$ a presentably symmetric monoidal $\infty$-category, the linear dual functor $(-)^\vee:\coalg_\O(\C)\op\rightarrow \alg_\O(\C)$ is a right adjoint, for any essentially small $\infty$-operad $\O$. Its left adjoint: \[(-)^\circ:\alg_\O(\C)\rightarrow \coalg_\O(\C)\op,\] is called the \emph{finite dual}. Our arguments above show that $A^\circ\simeq A^\vee$, as $\O$-coalgebras in $\C$, for any weak quasi-dualizable $\O$-algebra $A$ in $\C$. This follows by unicity of left adjoint functors.
\end{rem}

\section{Duality between THH and coTHH} \label{sec general duality between thh and cothh}

We show here one of our main results: there is a duality between topological Hochschild homology and topological coHochschild homology. Recall that we have introduced the notion of weak quasi-proper algebra in Definition \ref{defi: weak-quasi proper algebra} and quasi-proper coalgebra in Definition \ref{def quasi proper}.

\begin{thm} \label{thm most general thh cothh duality}
Let $\C$ be a closed symmetric monoidal $\infty$-category that admits geometric realizations and totalizations.
\begin{enumerate}[label=\upshape\textbf{(\roman*)}]
\item\label{item: main thm algebra part} Let $A$ be a \XXX-proper $\ee$-algebra in $\C$. Then:
\[
\Big(\thh^\C(A)\Big)^\vee \simeq \cothh^\C(A^\vee).
\]
\item\label{item: main thm coalgebra part} Let $C$ be a quasi-proper $\ee$-coalgebra in $\C$. Then:
\[
\cothh^\C(C)\simeq \Big( \thh^\C(C^\vee) \Big)^\vee.
\]
\end{enumerate}
\end{thm}

We shall prove the above theorem at the end of the section. We first make some observations.

\begin{rem} Even for $C$ a proper $\ee$-coalgebra, we have examples where:
\[
\Big(\cothh^\C(C)\Big)^\vee\not\simeq  \thh^\C(C^\vee),
\]
as shown in Remark \ref{rem: counter-example for dual of cothh}.
\end{rem}

\begin{rem} Let $A$ be a weak quasi-proper $\ee$-algebra in $\C$. Combining \ref{item: main thm algebra part} of Theorem \ref{thm most general thh cothh duality} above with the unit of the adjunction of Proposition \ref{prop: linear dual is self-adjoint}, we obtain a natural map in $A$, dashed in the diagram below:
\[
\begin{tikzcd}
\thh^\C(A) \ar{r} \ar[dashed]{dr} & \Big(\thh^\C(A)\Big)^{\vee\vee}\ar{d}{\simeq}\\
& \Big(\cothh^\C(A^\vee)\Big)^\vee.
\end{tikzcd}
\]
Similarly, let $C$ be a quasi-proper $\ee$-coalgebra in $\C$. Combining \ref{item: main thm coalgebra part} of Theorem \ref{thm most general thh cothh duality} above with the unit of the adjunction of Proposition \ref{prop: linear dual is self-adjoint}, we obtain a natural map in $C$, dashed in the diagram below:
\[
\begin{tikzcd}
\thh^\C(C^\vee) \ar{r}\ar[dashed]{dr} & \Big(\thh^\C(C^\vee)\Big)^{\vee \vee}\ar{d}{\simeq}\\
& \Big(\cothh^\C(C)\Big)^\vee.
\end{tikzcd}
\]
As mentioned in the previous remark, these maps are \emph{not} equivalences in general.
\end{rem}

 \begin{rem}
Let $C$ be a quasi-proper $\sph$-coalgebra.
Combining the natural map in the above remark with the trace map in $K$-theory, we obtain:
\[
\textup{K}(C^\vee)\longrightarrow \thh(C^\vee) \longrightarrow\Big(\cothh(C)\Big)^\vee.
\]
We therefore obtain a map by applying the Spanier-Whitehead dual:
\[
\cothh(C) \longrightarrow \textup{K}(C^\vee)^\vee,
\]
as the composite
 $\cothh(C)\longrightarrow \Big(\thh(C^\vee)\Big)^\vee \longrightarrow \textup{K}(C^\vee)^\vee$.  By adjointness, this provides a map
 \[
 \textup{K}(C^\vee) \wedge \coTHH(C) \longrightarrow \sph.
 \]
 \end{rem}
 
We now prove Theorem \ref{thm most general thh cothh duality}. It essentially follows from the linear dual assigning the cyclic bar construction of a weak quasi-proper algebra $A$:
\[
\begin{tikzcd}
 \cdots \ar[shift left=3]{r}\ar[shift right=3]{r}\ar[shift left]{r}\ar[shift right]{r} &A\ot A \ot A \ar{r}\ar[shift left=2]{r}\ar[shift right=2]{r} &A\ot A\ar[shift left]{r}\ar[shift right]{r} & A.
\end{tikzcd}
\]
to the cyclic cobar construction of the linear dual coalgebra $A^\vee$:
\[
\begin{tikzcd}
 \cdots & \ar[shift left=3]{l}\ar[shift right=3]{l}\ar[shift left]{l}\ar[shift right]{l} (A\ot A\ot A)^\vee & \ar{l}\ar[shift left=2]{l}\ar[shift right=2]{l} (A\ot A)^\vee & \ar[shift left]{l}\ar[shift right]{l}  A^\vee \ar[equals]{d}\\
 \cdots & \ar[shift left=3]{l}\ar[shift right=3]{l}\ar[shift left]{l}\ar[shift right]{l} A^\vee \ot A^\vee\ot A^\vee \ar{u}{\simeq}& \ar{l}\ar[shift left=2]{l}\ar[shift right=2]{l} A^\vee\ot A^\vee \ar{u}{\simeq} & \ar[shift left]{l}\ar[shift right]{l}  A^\vee.
\end{tikzcd}
\]
We make precise the above idea in the following lemma.

\begin{lem}\label{lem: bar=cobar on linear dual}
Let $\C$ be a closed symmetric monoidal $\infty$-category.
Let $A$ be a \XXX-proper $\ee$-algebra in $\C$.
Then we obtain an equivalence of cosimplicial objects in $\C$: 
\[\left(\bar^\C(A)\right)^\vee \simeq \cobar_\C\left(A^\vee\right).\]
\end{lem}

\begin{proof}
For $A$ a weak quasi-proper $\ee$-algebra in $\C$, we consider the subcategory $\C_A$ of $\C$ just as in the proof of Proposition \ref{prop: cond 1 implies dual of algebra is a coalgebra}. Notice that $A$ is also an $\ee$-algebra in $\C_A$.  Since the composition
\[
\begin{tikzcd}
\C_A \ar[hook]{r} & \C\ar{r}{(-)^\vee} & \C\op
\end{tikzcd}
\]
is strong symmetric monoidal, by Lemma \ref{lem: strong monoidal preserves cyclic bar}, we obtain an equivalence:
\[
\left(\bar^{\C_A}(A)\right)^\vee \simeq \bar^{\C\op}(A^\vee).
\]
By applying again Lemma \ref{lem: strong monoidal preserves cyclic bar} on the strong symmetric monoidal functor $\C_A\subseteq \C$, we have:
\[
\left(\bar^{\C_A}(A)\right)^\vee \simeq \left(\bar^{\C}(A)\right)^\vee.
\]
By definition, we also obtain:
\[
\bar^{\C\op}(A^\vee) \simeq \cobar_\C(A^\vee).
\]
Combining the above equivalences, we obtain the desired result.
\end{proof}

\begin{proof}[Proof of Theorem \ref{thm most general thh cothh duality}]
We first prove \ref{item: main thm algebra part}. By Lemma \ref{lem: bar=cobar on linear dual}, we have:
\[
\left( \bar^{\C}(A)\right)^\vee \simeq \cobar_\C(A^\vee).
\]
The functor $(-)^\vee:\C\op\rightarrow \C$, as a right adjoint (see Proposition \ref{prop: linear dual is self-adjoint}), preserves limits (or sends colimits to limits when regarded as a contravariant functor). Therefore, if we take the totalization of the cosimplicial objects on each side of the equivalence above we obtain the desired equivalence:
\[
\Big(\thh^\C(A)\Big)^\vee \simeq \cothh^\C(A^\vee).
\]

We now prove \ref{item: main thm coalgebra part}. Let $C$ be a quasi-proper $\ee$-algebra in $\C$. Then $C^\vee$ is a quasi-proper $\ee$-algebra and in particular weak quasi-proper, by Theorem \ref{thm: anti-equivalence between quasi-proper (co)algebras}. Therefore we can apply \ref{item: main thm algebra part} to $A=C^\vee$, and we obtain:
\[
\Big(\thh^\C(C^\vee)\Big)^\vee \simeq \cothh^\C(C^{\vee\vee}).
\]
By Theorem \ref{thm: anti-equivalence between quasi-proper (co)algebras}, we have $C^{\vee\vee}\simeq C$ as $\ee$-coalgebras in $\C$. Thus we obtain:
\[
 \cothh^\C(C^{\vee\vee})\simeq \cothh^\C(C).
\]
Combining the two equivalences, we have obtained the desired result.
\end{proof}

\section{Applications to the \texorpdfstring{$\infty$}{TEXT}-category of \texorpdfstring{$\hz$}{TEXT}-modules} \label{sec proof of cothh thh duality in chains}

In this section, we start with the proof Theorem \ref{thm cothh duality in chains}. This is an application of  Theorem \ref{thm most general thh cothh duality} to $HA$-modules where $A$ denotes a discrete commutative ring. After that, we prove Theorem \ref{thm equivalence of connective coalgebras and coconnective coalgebras in chains} which provides a contravariant equivalence of $\infty$-categories between algebras and coalgebras in (co)connective $HA$-modules with finite and free homology. We apply Theorem \ref{thm: anti-equivalence between quasi-proper (co)algebras} to prove Theorem \ref{thm equivalence of connective coalgebras and coconnective coalgebras in chains}. 

\subsection{Duality between THH and coTHH in \texorpdfstring{$HA$}{TEXT}-modules}\label{subsection duality between HH and coHH}
Let $A$ denote a discrete commutative  Noetherian ring for the rest of this section. Furthermore, we assume that $A$ has finite global dimension, i.e.\ there is an $d$ such that every $A$-module has a projective resolution of length at most $d$. Recall that we say an $HA$-module (or an $HA$-(co)algebra) $M$ is of finite type if $\pi_iM$ is a finitely generated $A$-module for every $i$. We restate Theorem \ref{thm cothh duality in chains} below.

\begin{thm}\label{thm restatement of cothh thh duality in chains}
Let $C$ be a  connective or coconnective $HA$-coalgebra of finite type where $A$ is as  above. Then there is an equivalence of $HA$-module spectra:
\[\cothh^{HA}(C) \simeq \left(\thh^{HA}(C^\vee)\right)^\vee.\]
\end{thm}

\begin{proof}
We apply Theorem \ref{thm most general thh cothh duality}. We need to show that $C$ is a quasi-proper $HA$-coalgebra.
It follows by Proposition \ref{prop cond 1 2 and cond 2 on c ot c imply quasi-proper} that it is sufficient to show that $C$ satisfies  Conditions \ref{cond 1} and \ref{cond 2} and $C^{\wdg_{HA}n}$  satisfies Condition \ref{cond 2} in $HA$-modules for every $n$.
This follows by Lemmas \ref{lem C and C ot C stsfy cndt 2} and \ref{lem C satisfies condition 1} below.
\end{proof}

\begin{rem}
In Remark 2.6 of \cite{HScothh},  Hess and Shipley claim a  duality relation between coHochschild homology and Hochschild homology in $k$-chain complexes where $k$ denotes a field. Their claim is an equivalence:
\[\left(\cohh^k_*(C)\right)^\vee \cong \HH^k_*(C^\vee),\]
for a  general coalgebra $C$ in $k$-chain complexes. However, we believe that this is not true in this generality and that  one needs finiteness and (co)connectivity conditions as in Theorem \ref{thm cothh duality in chains}. This is due to our Examples \ref{ex cond 2 doesnt imply cond 1} and \ref{ex counter to cond 1} showing that the linear dual functor is not strong monoidal without these assumptions.
\end{rem}

The rest of this subsection is devoted to the proof of Lemmas \ref{lem C and C ot C stsfy cndt 2} and \ref{lem C satisfies condition 1} below. For this, we begin with an identification of dualizable objects in $HA$-modules, see Proposition \ref{prop dualizable ha modules}. From Proposition \ref{prop dualizable ha modules}, we deduce that for a connective $C$ as in Theorem \ref{thm cothh duality in chains}, every Postnikov section of $C$ is dualizable and therefore satisfies the relevant conditions. In the proofs of Lemmas \ref{lem C and C ot C stsfy cndt 2} and \ref{lem C satisfies condition 1}, we consider the Postnikov section maps $C \to \tau_{\leq n}C$ for various $n$ to prove that $C$ also satisfies Conditions \ref{cond 1} and \ref{cond 2} and $C^{\wdg_{HA}n}$ satisfies Condition \ref{cond 2} as desired. For coconnective $C$, we use the connective cover functors instead of the Postnikov section functors and argue similarly.

Recall from Example \ref{ex: dualizable in R-modules}  
that dualizable objects in $HA$-modules are precisely what are called the perfect $HA$-modules, see  \cite[\RomanNumeralCaps{3}.7.9]{elmendorf2007rings} and \cite[3.2.3]{lurie2}.  An $HA$-module is said to be perfect if it lies in the smallest subcategory of the homotopy category of $HA$-modules that contains $HA$ and is closed under triangles and retracts. For instance, finite coproducts of shifted copies of $HA$ are examples of perfect $HA$-modules. 

Since $A$ is Noetherian of finite projective dimension, every finitely generated $A$-module $M$ admits  a finite length  resolution of finitely generated projective $A$-modules. This shows that for every finitely generated $A$-module $M$, $HM$ is perfect and therefore dualizable in $HA$-modules. Furthermore, we have the following proposition that characterize dualizable objects in $HA$-modules. 

\begin{prop}\label{prop dualizable ha modules}
Let $A$ be a discrete Noetherian ring of finite global dimension. 
An $HA$-module $N$ is dualizable if and only if $\pi_*N$ is finitely generated as a graded $A$-module, i.e.\ $\pi_iN$ is a finitely generated $A$-module for every $i$ and $\pi_iN\neq 0$ for only finitely many $i$.
\end{prop}

\begin{proof}
Let $N$ be a dualizable $HA$-module. This implies that $N$ can be obtained from $HA$ via taking finitely many triangles and retracts. The conclusion is immediate for the case $N=HA$ as $\pi_*HA = A$ is a finitely generated $A$-module. Therefore, it is sufficient to show that the property of having finitely generated homotopy is closed under retracts and triangles. 
The closeness under retracts follows form the fact that the codomain of a surjective map is finitely generated when the domain is finitely generated. For triangles, assume that there is a triangle
\[M \longrightarrow N \longrightarrow L,\]
in $HA$-modules where $M$ and $L$ have finitely generated homotopy groups. We need to show that $N$  also has finitely generated homotopy groups. From the induced long exact sequence, we get that $\pi_i N\neq 0$  for only finitely many $i$. Therefore it is sufficient to show that $\pi_iN$ is finitely generated for each $i$. The induced long exact sequence in homotopy can be broken into short exact sequences of the form: 
\[0 \longrightarrow Y \longrightarrow \pi_iN \longrightarrow Z \longrightarrow 0,\]
where $Z$ is a submodule of the finitely generated $A$-module $\pi_iL$. Since $A$ is Noetherian, $Z$ is also finitely generated. The $A$-module $Y$ is the image of the finitely generated $A$-module $\pi_iM$ under an $A$-module homomorphism and therefore $Y$ is also finitely generated. This proves that $\pi_iN$ is finitely generated. We conclude that $\pi_*N$ is a finitely generated graded $A$-module as desired.

Now we prove the other direction. Since dualizable $HA$-modules are closed under triangles, we get that an $HA$-module is dualizable if and only if one of its (de)suspensions is dualizable. Therefore, we can assume without loss of generality that $N$ is connective. For a given $HA$-module $N$ with finitely generated homotopy, let $h(N)$ be the largest $i$ for which $\pi_iN \neq 0$. We argue inductively on $h(N)$.

For the base case of our induction, we have $N=HM$, for some finitely generated $A$-module $M$. We already argued that $N$ is a dualizable object in this situation. 

Assume that every connective $A$-module $M$ with $h \geq h(M)$ and finitely generated homotopy is dualizable. Let $N$ be an $HA$-module with finitely generated $\pi_*N$ and $h(N)= h+1$. We need to show that $N$ is dualizable. 
Since $\pi_0 N$ is finitely generated, we have a map: 
\[\bigvee_{j = 1}^m HA \longrightarrow N,\]
that induces a surjective map on degree $0$ homotopy where the domain denotes a finite coproduct. 
Let $N^\prime$ be defined by the cofiber sequence: 
\[\bigvee_{j = 1}^m HA \longrightarrow N \longrightarrow N^\prime.\]
Using the long exact sequence induced by this cofiber sequence, one observes that $\pi_0N' = 0$ and: 
\[\pi_iN'= \pi_i N = 0,\]
for every $i > h(N)=h+1$ and $i<0$. Since $A$ is Noetherian, every submodule of a finitely generated $A$-module is also finitely generated. This fact, together with the aforementioned long exact sequence implies that $\pi_i N'$ is finitely generated for every $i$. It follows by our induction hypothesis that $\Sigma^{-1}N'$ is dualizable because $h(\Sigma^{-1}N')\leq h $. Therefore, $N'$ is also dualizable and it is generated by $HA$ under triangles and retracts. The triangle above proves that $N$ is also generated by $HA$ under triangles and retracts. In other words, $N$ is also dualizable.
\end{proof}

For the rest of this section, let $d$ denote the global dimension of $A$.  We say a  spectrum $E$ is \emph{$n$-connective} if $\pi_iE =0$ for every $i < n$. Similarly, we say $E$ is \emph{$n$-coconnective} if $\pi_iE =0$ for every $i > n$. Before we start the proof of Lemmas \ref{lem C and C ot C stsfy cndt 2} and \ref{lem C satisfies condition 1}, we need to prove the following technical lemmas.

\begin{lem}\label{lem dualization and connectivity}
If $N$ is an  $n$-connective $HA$-module, then $N^\vee$ is an $-n$-coconnective $HA$-module. If $N$ is an $n$-coconnective $HA$-module, then $N^\vee$ is a $-(n+d)$-connective $HA$-module. 
\end{lem}

\begin{proof}
For this, we use the Ext spectral sequence of \cite[\RomanNumeralCaps{4}.4.1]{elmendorf2007rings} to compute the mapping space defining $N^\vee$. The $E_2$-page of this spectral sequence is given by: 
\[E_2^{p,q} =  \Ext^{p,-q}_A(\pi_*N, A),\]
and it abuts to: 
\[\pi_{-(p+q)}\left([N,A]\right),\]
where $[-,-]$ denotes the internal mapping   spectrum in $HA$-modules and the Ext groups $\Ext^{*,-q}_A(-, -)$  are computed by considering the maps that decrease the internal degree by $q$. For the first statement, note that $\pi_*N$ has a projective resolution consisting of graded modules that are trivial below degree $n$ and this ensures that $E_2^{*,q}=0$ for $q < n$.  Since $p \geq 0$, we have $q<n$ whenever $-(p+q)>-n$ and therefore  $E_2^{p,q}=0$ whenever $-(p+q)>-n$. This proves the first statement.

Since $A$ has global dimension $d$, $\pi_*N$ has a projective resolution of length $d$ in graded $A$-modules. This shows that $E_2^{p,*}= 0$
for every $p>d$. Since $N$ is $n$-coconnective,  $E_2^{*,q}\neq 0$ implies $q\leq n$. Therefore the  possible non-trivial entries contributing to the homotopy group in degree $-(p+q)$ with   $-(p+q)<-(n+d)$ should also satisfy $q \leq n$. But in this situation, we have $p>d$ and therefore $E_2^{p,q}=0$. In other words, $E_2^{p,q}=0$ whenever $-(p+q)<-(n+d)$.
\end{proof}

\begin{lem}\label{lem tensors and connectivity}
Let $M$ and $N$ be $A$-modules. If $M$ is $m$-connective and $N$ is $n$-connective then $M \wdg_{HA}N$ is $(m+n)$-connective. If $M$ and $N$ are $m$-coconnective and $n$-coconnective respectively, then $M\wdg_{HA} N$ is $(m+n +d)$-coconnective.
\end{lem}
\begin{proof}
In this case, we use the K\"unneth spectral sequence in \cite[\RomanNumeralCaps{4}.4.1]{elmendorf2007rings} to compute $M \wdg_{HA}N$. The $E^2$-page is given by: 
\[E^2_{p,q}= \textup{Tor}_{p,q}^A(\pi_*M,\pi_*N) \Longrightarrow \pi_{p+q}(M \wdg_{HA} N).\]
In the first situation, there is a flat resolution of $\pi_*M$  given by graded $A$-modules that are trivial below degree $m$. Tensoring this resolution with $\pi_*N$ gives a resolution that is trivial in degrees  below $m+n$. Therefore $E^2_{p,q}=0$ for $q<m+n$. This gives the desired result. 

For the second statement, note that $\pi_*M$ admits a projective and therefore a flat resolution of length at most $d$. Therefore, $E^2_{p,*}=0$ for every $p > d$. Due to the coconnectivity assumptions, we also have $E^2_{*,q}=0$ whenever $q>m+n$. This provides the second statement in the lemma.
\end{proof}

\begin{lem} \label{lem C and C ot C stsfy cndt 2}
If $C$ is a (co)connective finite type $HA$-module, then $C^{\wdg_{HA}n}$ satisfies Condition \ref{cond 2} for every $n\geq0$. In particular, $C$ satisfies Condition \ref{cond 2}.
\end{lem}
\begin{proof}

We start with the proof of the case  $n=1$. Let $C$ be a connective and finite type $HA$-module. To prove that $C$ satisfies Condition \ref{cond 2}, we need to show that the natural map: 
\[\eta_C: C \longrightarrow C^{\vee\vee},\]
is a weak equivalence. We are going to show that this map induces an isomorphism on each homotopy group. Let $i \in \z$, to show that $\pi_i \eta_C $ is an isomorphism, we use the cofiber sequence:
\begin{equation} \label{eq first cofib sequence}
    \tau_{\geq n}C \longrightarrow C \longrightarrow \tau_{\leq n-1}C.
\end{equation}
Considering the natural transformation $\eta$ on this cofiber sequence, we obtain the following diagram where the vertical sequences are the canonical cofiber sequences:
\begin{equation} \label{diag cofib seq and the natural transformation}
\begin{tikzcd}
\tau_{\geq n}C \ar{r} \ar{d} & (\tau_{\geq n}C)^{\vee \vee} \ar[d]\\
C \ar{r} \ar[d] & C^{\vee \vee} \ar[d]\\
\tau_{\leq n-1}C \ar[r]&(\tau_{\leq n-1}C)^{\vee\vee}.  
\end{tikzcd}
\end{equation}
Let $n=i+d+1$. Due to Proposition \ref{prop dualizable ha modules}, $\tau_{\leq i+d}C$ is dualizable and therefore Corollary \ref{cor dualizable satisfy cond 1 and 2} implies that the bottom horizontal arrow is an equivalence. Since $\tau_{\geq i+d+1}C$ is $(i+d+1)$-connective, $(\tau_{\geq i+d+1}C)^\vee$ is $-(i+d+1)$-coconnective due to Lemma \ref{lem dualization and connectivity}. This, together with Lemma \ref{lem dualization and connectivity}  shows that $(\tau_{\geq i+d+1}C)^{\vee\vee}$ is $(i+1)$-connective. In particular, both sides of the top horizontal arrow are $(i+1)$-connective. The map of long exact sequences induced by the diagram of cofiber sequences above shows that $C \to C^{\vee\vee}$ induces an isomorphism in degree $i$ homotopy groups. This proves that $C$ satisfies Condition \ref{cond 2}.

For the case where $C$ is coconnective, we use the cofiber sequence in \eqref{eq first cofib sequence} for $n=i-d$.
In this case, $\tau_{\geq i-d}C$ is dualizable  due to Proposition \ref{prop dualizable ha modules} and therefore the top horizontal  arrow in Diagram \eqref{diag cofib seq and the natural transformation} is an equivalence. Also,  $\tau_{\leq i-d-1}C^{\vee\vee}$ is $(i-1)$-coconnective due to Lemma \ref{lem dualization and connectivity}. In particular, both sides of the bottom horizontal arrow are $(i-1)$-coconnective. It follows by considering the map of long exact sequences induced by Diagram \eqref{diag cofib seq and the natural transformation} that $\eta_C$ induces an isomorphism on degree $i$ homotopy. 

Now we prove that for every $C$ as in Theorem \ref{thm cothh duality in chains}, $C^{n\wdg_{HA}}$ also satisfies Condition \ref{cond 2} for every $n\geq 0$. For simplicity, we write $C^n$ for the $n$-fold smash product $C^{\wdg_{HA} n}.$ Note that $n=0$ case is satisfied trivially because $HA$ is dualizable. The $n=1$ case is what we prove above. We proceed by induction. Assume that this is true for some $(n-1)\geq 1$, we need to show that the natural transformation:
\[
\eta_{C^n}\co C^n \longrightarrow (C^n)^{\vee \vee},
\]
is a homotopy isomorphism, i.e.\ a weak equivalence. Let $C$ be connective and finite type. Given $i\in \z$, let $\ell=i+d$, there is a cofiber sequence: 
\[(\tau_{\geq \ell+1}C)\wdg_{HA} C^{n-1} \longrightarrow C \wdg_{HA} C^{n-1} \longrightarrow (\tau_{\leq \ell}C) \wdg_{HA} C^{n-1}. \]
Due to Lemma \ref{lem tensors and connectivity}, $(\tau_{\geq \ell+1}C)\wdg_{HA} C^{n-1}$ is $(i+d+1)$-connective and it follows from Lemma \ref{lem dualization and connectivity} that  $((\tau_{\geq \ell+1}C)\wdg_{HA} (C^{n-1}))^{\vee\vee}$ is $(i+1)$-connective. We apply the natural transformation $\eta$ to this cofiber sequence and consider the induced map of long exact sequences. This shows that in order to prove  $\pi_i\eta_{C^n}$ is an isomorphism, it is sufficient to show that: 
\[
\eta_{(\tau_{\leq \ell}C) \wdg_{HA} C^{n-1}}\co (\tau_{\leq \ell}C) \wdg_{HA} C^{n-1} \longrightarrow ((\tau_{\leq \ell}C) \wdg_{HA} C^{n-1})^{\vee\vee},\]
is an equivalence. Note that this map is an equivalence when $\tau_{\leq \ell}C$ is replaced by $HA$ by our induction hypothesis. Furthermore, both the domain and the codomain  of this natural transformation preserves triangles and retracts in the first factor of the smash product. Due to Proposition \ref{prop dualizable ha modules}, $\tau_{\leq \ell}C$ is dualizable and therefore it is generated by $HA$ under triangles and retracts. This shows that $\eta_{(\tau_{\leq \ell}C) \wdg_{HA} C^{n-1}}$ is also an equivalence.

We also need to show that $C^n$ satisfies Condition \ref{cond 2} under the induction hypothesis whenever $C$ is coconnective and finite type. Fix an integer $i$ and let $\ell$ be a negative integer such that:
\[\ell+nd<i-1 \textup{\ and \ } \ell+(n-1)d <i-1.\]
We need to show that $\pi_i\eta_{C^n}$ is an isomorphism. We consider the similar cofiber sequence:
\[(\tau_{\geq \ell+1}C)\wdg_{HA} C^{n-1} \longrightarrow C \wdg_{HA} C^{n-1} \longrightarrow (\tau_{\leq \ell}C) \wdg_{HA} C^{n-1}. \]
The right hand side is $(\ell+(n-1)d)$-coconnective and therefore its double dual is $(\ell+nd)$-coconnective. Therefore both the right hand side and its double dual are $(i-1)$-coconnective. Applying the natural transformation  $\eta$ to this cofiber sequence, one sees that it is sufficient to show that $\eta_{(\tau_{\geq \ell+1} C )\wdg_{HA} C^{n-1}}$ is an equivalence. 
This follows as before due to the fact that  $\tau_{\geq \ell+1} C$ is dualizable.
\end{proof}

\begin{lem}\label{lem C satisfies condition 1}
If $C$ is a (co)connective finite type $HA$-module then $C$ is weak quasi-dualizable, i.e.\ $C$ satisfies Condition \ref{cond 1}.
\end{lem}
\begin{proof}
As before, we write $C^n$ for the $n$-fold smash product $C^{\wdg_{HA}n}.$ By induction, it is sufficient to show that the lax symmetric monoidal structure map: 
\[\varphi \co C^\vee   \wdg_{HA}(C^{n-1})^\vee \to (C^n)^\vee,\]
is an equivalence for every $n$.

We first prove this for the case where $C$ is a connective and finite type $HA$-module. For  a given integer $i$, we need to show that the map above induces an isomorphism in degree $i$ homotopy. Let $\ell$ be an integer such that $-\ell+d<i-1$ and $-\ell <i-1$. There is a cofiber sequence: 
\[\tau_{\geq \ell}C \to C \to \tau_{\leq \ell-1} C.\]
Recall the that dualization functor is a contravariant functor. From this cofiber sequence and the natural transformation $\varphi$, we obtain the following diagram:
\begin{equation}\label{diag cond 1 for finite type things}
\begin{tikzcd}[column sep=small]
(\tau_{\leq \ell-1}C)^\vee \wdg_{HA} (C^{n-1})^\vee \ar{r} \ar{d} & ((\tau_{\leq \ell-1}C) \wdg_{HA} C^{n-1})^\vee\ar[d]\\
C^\vee \wdg_{HA}(C^{n-1})^\vee\ar{r} \ar[d] & (C \wdg_{HA} C^{n-1})^\vee \ar[d].\\
(\tau_{\geq \ell}C)^\vee \wdg_{HA} (C^{n-1})^\vee \ar[r]&((\tau_{\geq \ell}C) \wdg_{HA }C^{n-1})^\vee.
\end{tikzcd}
\end{equation}
The vertical lines above are cofiber sequences and the horizontal maps are given by lax monoidal structure map of the dualization functor. Since $\tau_{\geq \ell}C$ is $\ell$-connective and $C$ is $0$-connective, $(\tau_{\geq \ell}C)^\vee$ is $-\ell$-coconnective and $(C^{n-1})^\vee$ is $0$-coconnective due to Lemmas \ref{lem dualization and connectivity} and \ref{lem tensors and connectivity}. This, together with Lemma \ref{lem tensors and connectivity} shows that the bottom left corner is $(-\ell+d)$-coconnective; in particular, it is $(i-1)$-coconnective. Similarly, the bottom right corner is $-\ell$-coconnective and therefore $(i-1)$-coconnective. 

The map of long exact sequences induced by the diagram above shows that it is sufficient to prove that the top horizontal arrow is an equivalence. 
Note that $\tau_{\leq \ell-1}C$ is dualizable in $HA$-modules due to Proposition \ref{prop dualizable ha modules}. Therefore, $\tau_{\leq \ell-1}C$ can be obtained from $HA$ via finitely many triangles and retracts. Note that the arrow above is an equivalence if we replace  $\tau_{\leq \ell-1}C$ by $HA$ and that both the domain and the codomain of the top horizontal arrow above preserve triangles and retracts in the first variable. Since $\tau_{\leq \ell-1}C$  can be built from $HA$ via triangles and retracts, this shows that the top arrow in the diagram above is an equivalence.

Finally, we need to show that an $HA$-module $C$ satisfies Condition 1 whenever it is coconnective and finite type. Fix $i$ and let $\ell$ be an integer such that: 
\[-\ell+1-nd>i+1.\] We need to show that $\pi_i\varphi$ is an isomorphism. Again, we consider Diagram (\ref{diag cond 1 for finite type things}). In this case, $\tau_{\leq \ell-1} C$ is $(\ell-1)$-coconnective and $C$ is $0$-coconnective. Therefore, $(\tau_{\leq \ell-1} C)^\vee$
is $(-\ell+1-d)$-connective and $C^{n-1}$ is $(n-2)d$-coconnective. Also, $(C^{n-1})^\vee$ is $-(n-1)d$-connective. Therefore, the domain of the top horizontal map in Diagram (\ref{diag cofib seq and the natural transformation}) is 
 and $(-\ell+1-dn)$-connective. Similarly, the codomain of the top  horizontal map is also $(-\ell+1-dn)$-connective. In particular, both the domain and the codomain of the top horizontal map are $(i+1)$-connective.  Therefore, it is sufficient to show that the bottom horizontal map is an equivalence.
 
This follows in a way similar to the connective case. Note that $\tau_{\geq \ell}C$ is dualizable due to Proposition \ref{prop dualizable ha modules}. The bottom horizontal map is an equivalence if $\tau_{\geq \ell}C$ is replaced by $HA$. The domain and the codomain of this map preserves triangles and retracts in the first variable and $\tau_{\geq \ell}C$  can be built from $HA$ via retracts and triangles. Therefore, the bottom horizontal map is also an equivalence as desired.
\end{proof}

\subsection{A Duality Between Algebras and Coalgebras in \texorpdfstring{$\hz$}{TEXT}-modules}\label{subsection duality between algebras and coalgebras in chains}

In this section, we prove a duality result between algebras and coalgebras in the $\infty$-category of chain complexes over a general discrete commutative ring $B$. For the theorem below, let 
\[
\alg_{\O}(\Mod_{HB})_{\mathsf{fft}} \textup{  and  } \coalg_{\O}(\Mod_{HB})_{\mathsf{fft}},
\]
denote the full subcategory of finite type $\O$-algebras in $HB$-modules  whose homotopy groups are  free $B$-modules and the full subcategory of finite type $\O$-coalgebras in $HB$-modules whose homotopy groups are free $B$-modules respectively. Furthermore, ${}_{\geq 0}(-)$ and ${}_{\leq0}(-)$ denotes the restriction to the full subcategories of connective and coconnective and objects respectively. 

\begin{thm}\label{thm equivalence of connective coalgebras and coconnective coalgebras in chains}
There are equivalences of $\infty$-categories: 
\[{}_{\geq 0}\coalg_{\O}(\Mod_{HB})_{\mathsf{fft}}^{\mathsf{op}} \simeq {}_{\leq0}\alg_{\O}(\Mod_{HB})_{\mathsf{fft}},\]
and: 
\[ {}_{\leq0}\coalg_{\O}(\Mod_{HB})_{\mathsf{fft}}^{\mathsf{op}} \simeq  {}_{\geq 0}\alg_{\O}(\Mod_{HB})_{\mathsf{fft}},\]
given by the dualization functor $(-)^\vee$ in all directions.
\end{thm}

To prove this theorem, we use Theorem \ref{thm: anti-equivalence between quasi-proper (co)algebras} for the case $\C = \Mod_{HB}$. This shows that there is an equivalence of $\infty$-categories:
\begin{equation}\label{eq equivalence of infty categories restated}
    ((\coalg_\O(\Mod_{HB}))_\qdual)\op \simeq (\alg_\O(\Mod_{HB}))_\qdual,
\end{equation}
given in both directions by the dualization functor in $HB$-modules. We show that the equivalences of $\infty$-categories given in Theorem \ref{thm equivalence of connective coalgebras and coconnective coalgebras in chains} are given by restrictions of the equivalence above to the mentioned subcategories. The lower script $\qdual$ in \eqref{eq equivalence of infty categories restated} denotes the restriction to the quasi-proper objects, see Definition \ref{def quasi proper}. Recall that an $\O$-coalgebra $C$ in $HB$-modules is quasi-proper if $C$ satisfies Condition \ref{cond 2}
 and $C^\vee$ satisfies Condition \ref{cond 1}. On the other hand, an $\O$-algebra $X$ in $HB$-modules is said to be quasi-proper if it satisfies Conditions \ref{cond 1} and \ref{cond 2} in $HB$-modules.

 \begin{lem}\label{lem free things satisfy cond 1 and cond 2}
 Let $M$ be a (co)connective $HB$-module whose homotopy groups are finite dimensional free $B$-modules. In this situation, $M$ satisfies Conditions \ref{cond 1} and \ref{cond 2}. Furthermore, $M^\vee$ also satisfies Condition \ref{cond 1}. In other words, $M$ is quasi-dualizable and $M^\vee$ is weak-quasi dualizable in $HB$-modules.
 \end{lem}
 \begin{proof}
 We first prove this for the case where $M$ is connective. Note that the homotopy classes of maps of $HB$-modules with free homotopy groups are given by graded $B$-module maps of the corresponding homotopy groups. This follows by the Ext spectral sequence computing mapping spectra \cite[\RomanNumeralCaps{4}.4.1]{elmendorf2007rings}. The Ext spectral sequence also shows that $M$ is a wedge of suspensions of $HB$. One could also see this by choosing maps $\Sigma^k HB \to M$ for every basis element in $\pi_kM$ and taking a coproduct over these maps. 
 
 In particular, the Postnikov section $\tau_{\leq n}M$ is a wedge of finitely many suspensions of $HB$. Since dualizable $HB$-modules are those that can be obtained from $HB$ via finitely many triangles and retracts \cite[3.2.3]{lurie2}, this shows that $\tau_{\leq n}M$ is dualizable for every $n$. In particular, $\tau_{\leq n}M$ and $(\tau_{\leq n}M)^\vee$ satisfies Conditions \ref{cond 1} and \ref{cond 2}, see Corollary \ref{cor dualizable satisfy cond 1 and 2}. 
 
 At this point, we could argue as in the proof of Lemmas \ref{lem C and C ot C stsfy cndt 2} and \ref{lem C satisfies condition 1} but this situation is simpler. For instance, to show that the map
 \[M\to M^{
 \vee\vee}\]
 is an isomorphism in degree $i$ homotopy, we use the map $M \to \tau_{\leq i+1}M$ to compare the map above with the map: 
 \[\tau_{\leq i+1}M \to \tau_{\leq i+1}M^{\vee\vee}.\]
 Since dualization in these cases results in dualization of graded $B$-modules at the level of homotopy, these two maps agree on degree $i$ homotopy and the second map is an equivalence because $\tau_{\leq i+1}M$ is dualizable. This shows that $M$ satisfies Condition \ref{cond 2}. 
 
To see that $M$ satisfies Condition \ref{cond 1}, note that smash products  of $HB$-modules with free homotopy are given by tensor products over $B$ at the level of homotopy groups. This follows by the K\"unneth spectral sequence \cite[\RomanNumeralCaps{4}.4.1]{elmendorf2007rings}. Again, to show that the relevant map
 \[(M^\vee)^{n \wdg_{HB}} \to (M^{n\wdg_{HB}})^\vee\]
gives an equivalence in degree $i$ homotopy, we use the map $M\to \tau_{\leq i+1} M$ to compare this with the map
 \[((\tau_{\leq i+1} M)^\vee)^{n \wdg_{HB}} \to ((\tau_{\leq i+1} M)^{n\wdg_{HB}})^\vee.\] 
 which is an equivalence because $\tau_{\leq i+1} M$ is dualizable. Observing that these maps agree on degree $i$ homotopy, we deduce that the first map is also an equivalence and that $M$ satisfies Condition \ref{cond 1}. 
 
 Using the Ext spectral sequence one more time, one observes that $M^\vee$ is a coconnective $HB$-module where each homotopy group is a  finite dimensional free $B$-module. By using the connective cover functor $\tau_{\geq i} -$ instead of the Postnikov section functor $\tau_{\leq i} -$, one argues as before to see that $M^\vee$ also satisfies Condition \ref{cond 1}. The coconnective case of the lemma also follows similarly. 
 \end{proof}
 
 \begin{proof}[Proof of Theorem \ref{thm equivalence of connective coalgebras and coconnective coalgebras in chains}]
 
 We start with the proof of the first equivalence in the theorem. We show that this equivalence is a restriction of the equivalence  in \eqref{eq equivalence of infty categories restated} to a full subcategory. For this, it is sufficient to show that  that the $\infty$-category
 \[{}_{\geq 0}\coalg_{\O}(\Mod_{HB})_{\mathsf{fft}}^{\mathsf{op}}\]
 is a full subcategory of  $((\coalg_\O(\Mod_{HB}))_\qdual)\op$ and the essential image of this full subcategory is given by \[{}_{\leq0}\alg_{\O}(\Mod_{HB})_{\mathsf{fft}},\] 
 under the dualization functor. 
 For the first part, we need to show that every $C$ in ${}_{\geq 0}\coalg_{\O}(\Mod_{HB})_{\mathsf{fft}}^{\mathsf{op}}$ is quasi-proper, i.e.\ $C$ satisfies Condition \ref{cond 2} and $C^\vee$ satisfies Condition \ref{cond 1}. This follows by Lemma \ref{lem free things satisfy cond 1 and cond 2}. 
 
 What is left is to show that the essential image of  ${}_{\geq 0}\coalg_{\O}(\Mod_{HB})_{\mathsf{fft}}^{\mathsf{op}}$ under the dualization functor is ${}_{\leq0}\alg_{\O}(\Mod_{HB})_{\mathsf{fft}}$. First, we show that for every $C$ in ${}_{\geq 0}\coalg_{\O}(\Mod_{HB})_{\mathsf{fft}}^{\mathsf{op}}$, $C^\vee$ lies in ${}_{\leq0}\alg_{\O}(\Mod_{HB})_{\mathsf{fft}}$. This follows by the Ext spectral sequence showing that on an $HB$-module whose homotopy groups are  free $B$-modules, the dualization functor results in graded $B$-module dualization at the level of homotopy groups. Given $X$ in ${}_{\leq0}\alg_{\O}(\Mod_{HB})_{\mathsf{fft}}$, we need to show that $X= C^\vee$ for some $C$ in ${}_{\geq 0}\coalg_{\O}(\Mod_{HB})_{\mathsf{fft}}^{\mathsf{op}}$. For this, note that $X$ is quasi-dualizable due to Lemma \ref{lem free things satisfy cond 1 and cond 2}, therefore $X$ lies in: 
 \[(\alg_\O(\Mod_{HB}))_\qdual.\]
 The equivalence in \eqref{eq equivalence of infty categories restated} shows that $(X^\vee)^\vee \simeq X$ and that $X^\vee$ is an $\O$-coalgebra in $HB$-modules. It follows by inspection on homotopy groups that  $X^\vee$ lies in ${}_{\geq 0}\coalg_{\O}(\Mod_{HB})_{\mathsf{fft}}^{\mathsf{op}}$ as desired. This finishes the proof of the first equivalence in the theorem.
 The proof of the second equivalence follows in the same way.
 \end{proof}

\section{Coalgebras in Spectra}\label{sec examples of coalgebras in spectra}

We have introduced the definition of topological coHochschild homology associated to an $\ee$-coalgebra in any symmetric monoidal $\infty$-category. We are interested here in $\ee$-coalgebras in the $\infty$-category of spectra $\sp$ or more generally in $\ee$-coalgebra in $R$-modules, where $R$ is a commutative ring spectrum.

In symmetric monoidal model categories of spectra, such as symmetric spectra, an $\mathbb{S}$-coalgebra structure is quite a restrictive feature on a spectrum. However, unlike for algebras, the category of strictly coassociative and counital $\mathbb{S}$-coalgebras do not represent the $\mathbb{A}_\infty$-coalgebras in $\sp$. See \cite{perouxshipley} and \cite{perouxDKloc}. The main point is that the model categories only capture coalgebras that behave as the cocommutative $\sph$-coalgebra $\Sigma^\infty_+ X$ with comultiplication induced by the diagonal:
\[
X_+ \longrightarrow (X\times X)_+\cong X_+\wedge X_+,
\]
where $X$ is a space and $X_+$ is the free pointed space on $X$. In particular, strictly coassociative and counital $\sph$-coalgebras do not capture the Spanier-Whitehead duality between $\ee$-algebras and $\ee$-coalgebras in finite spectra of Corollary \ref{cor: antiequivalence proper (co)algebras}.

Nonetheless, we show in this section that an $\ee$-coalgebra structure on a spectrum $C$ is also a restrictive feature. If $R$ is a commutative ring spectrum, then an $\ee$-algebra in $R$-modules in $\sp$ is also an $\ee$-algebra in $\sp$. In other words, we have that spectra in $\alg_\ee(\sp)$ require less structure than spectra in $\alg_\ee(\Mod_R(\sp))$.
This is not true for coalgebras. In fact, we can say that spectra in $\coalg_\ee(\sp)$ have a more restrictive structure than spectra in $\coalg_\ee(\Mod_R(\sp))$. For instance, if $C$ is an $R$-module endowed with an $\ee$-coalgebra structure in $\sp$, then $C$ is an $\ee$-coalgebra in $R$-modules in $\sp$.

One of the main restrictive feature for an $\ee$-algebra $C$ in $\sp$ is the requirement of a map $C\rightarrow \sph$ that is counital (up to higher homotopy).

\begin{thm} \label{thm connective  coalgebras are suspension spectra}
Let $C$ be a connective $\ee$-coalgebra in the $\infty$-category of spectra. If $\pi_0 C \neq 0$ then the $p$-localization of $C$ contains the $p$-local sphere spectrum $\Sp_{(p)}$ as a retract in the stable homotopy category for every $p\gg0$. 
\end{thm}

\begin{proof}
Since $\C$ is an $\ee$-coalgebra in $\sp$, it is in particular a coassociative counital coalgebra up to homotopy: i.e. it is a strictly coassociative counital coalgebra in the homotopy category of $\sp$.
Let $\Delta$ denote the comultiplication map and $\epsilon$ denote the counit map of $C$ in the homotopy category of $\sp$ respectively. The following composite is homotopic to the identity map on $C$:
\begin{equation} \label{eq non example diagram}
    C \xrightarrow{\Delta} C \wedge C \xrightarrow{\epsilon \wedge id_C} \Sp \wedge C \simeq C,
\end{equation}
where the smash product denotes the derived smash product. Since $C$ is connective, by the K\"unneth spectral sequence  \cite[\rom{4}.4.1]{elmendorf2007rings} we have:  
\[\pi_0(C \wedge C) \cong \pi_0C \otimes \pi_0 C \textup{\ and\ } \pi_0(\Sp \wedge C) \cong \pi_0\Sp \otimes \pi_0C.\]
We apply $\pi_0(-)$ to \eqref{eq non example diagram} and due to the functoriality of the K\"unneth spectral sequence, we obtain the following composite: 
\[\pi_0C \to \pi_0 C \otimes \pi_0 C \xrightarrow{\pi_0 \epsilon \otimes \pi_0 id_C} \z \otimes \pi_0C \cong \pi_0C. \]
If $\pi_0 \epsilon$ is the trivial map, it follows that the composite above is trivial but this contradicts the fact that this composite is the identity map and that $\pi_0C \neq 0$. Therefore $\pi_0 \epsilon$ is non-trivial. Let $m$ be the smallest positive integer in the image of $\pi_0 \epsilon$ and let 
\[c \co \Sp \to C, \] 
denote a map whose image under $\pi_0 \epsilon$ is $m$. If $m = 1$, then this makes sure that $C$ contains the sphere spectrum $\Sp$ as a retract in the stable homotopy category through the composite: 
\[\Sp \xrightarrow{c} C \xrightarrow{\epsilon} \Sp.\]
For $m>1$, let $p$ be a prime greater than $m$. In this situation, $m$ is a unit in $\mathbb{Z}_{(p)}$ and therefore a unit in $\pi_0 (\Sp_{(p)})$. Therefore, there is a map $m^{-1} \co \Sp_{(p)} \to \Sp_{(p)}$ representing $m^{-1}$ in $\pi_0 (\Sp_{(p)})$. We obtain the following composite: 
\[\Sp_{(p)} \xrightarrow{c} C_{(p)} \xrightarrow{\epsilon} \Sp_{(p)} \xrightarrow{m^{-1}}\Sp_{(p)},\]
which is the identity map of $\Sp_{(p)}$ (in the homotopy category) as desired. 
In particular, this implies that $\pi_*C_{(p)}$ contains the $p$-local stable homotopy groups of spheres for every $p>m$. 
\end{proof}

By inspection on the homotopy groups of various spectra, we obtain the following.

\begin{cor}
The spectra $ko$, $ku$, $MU$, connective covers of the Morava E-theory $E_n$ and $K(n)$ are not $\ee$-coalgebras in the $\infty$-category of spectra $\sp$. For a discrete commutative ring $A$, the Eilenberg-Mac Lane spectrum $HA$ is not an $\ee$-coalgebra in $\sp$.
\end{cor}

\begin{ex}
Let $R$ be a commutative ring spectrum. Let $\mathcal{S}$ be the $\infty$-category of spaces. Then the functor
\[ 
R\wedge -: \mathcal{S}\longrightarrow \sp
\]
is strong symmetric monoidal. Here $\mathcal{S}$ is endowed with its Cartesian symmetric monoidal structure. In particular, for any space $X$, we get that $R\wedge X_+$ is a cocommutative $R$-coalgebra. 
For instance, for $R=\sph$, we have that $\Sigma^\infty_+ X$ is a cocommutative $\sph$-coalgebra in $\sp$.
\end{ex}

\begin{ex}\label{ex: hess-shipley}
From \cite[3.6]{HScothh}, we say a space is \emph{EMSS-good} if $X$ is connected and $\pi_1 X$ acts nilpotently on $H_i(\Omega X; \mathbb{Z})$ for all $i$. In particular, every simply connected space is EMSS-good. If $X$ is EMSS-good, then it was shown in \cite[3.7]{HScothh} that
\[
\cothh(\Sigma^\infty_+ X)\simeq \Sigma^\infty_+ \mathscr{L}X,
\]
where $\mathscr{L}X$ is the free loop space on $X$.
If $X$ is a finite CW-complex, then $\Sigma^\infty_+ X$ is dualizable in $\sp$, and by Theorem \ref{thm most general thh cothh duality}, we get:
\[
\cothh(\Sigma^\infty_+ X)\simeq \left(\thh((\Sigma^\infty_+ X)^\vee)\right)^\vee. 
\]
Thus if $X$ is a finite CW-complex and EMSS-good, we obtain:
\[
\cothh(\Sigma^\infty_+ X)\simeq \left(\thh((\Sigma^\infty_+ X)^\vee)\right)^\vee\simeq \Sigma^\infty_+ \mathscr{L}X.
\]
The last equivalence was proved for $X$ a finite CW-complex and simply connected in \cite{kuhn} and \cite{cary}. {Thus, our result generalizes the last equivalence above to EMSS-good finite CW-complexes.}
\end{ex}

\begin{ex}\label{ex: dual of smooth manifold}
Let $M$ be a compact smooth manifold. It is homotopic to a finite CW-complex and thus $\Sigma^\infty_+ M$ is dualizable in $\sp$. Then by Atiyah duality (see \cite[3.3]{atiyah}), the Spanier-Whitehead dual of the spectrum $\Sigma^\infty_+ M$ is the Thom spectrum $M^{-\tau}$ of its stable normal bundle. {The geometric construction of  $M^{-\tau}$ equips it with the structure of a commutative ring spectrum and, by \cite{ralph}, this structure agrees with the induced commutative ring structure on the Spanier-Whitehead dual of $\Sigma^\infty_+ M$ from Corollary \ref{cor: antiequivalence proper (co)algebras}.} Moreover, by Theorem \ref{thm most general thh cothh duality}, we obtain
\[
\cothh(\Sigma^\infty_+ M) \simeq \left(\thh(M^{-\tau})\right)^\vee.
\]
If we further assume that $M$ is EMSS-good, then we obtain by Example \ref{ex: hess-shipley}:
\[
\cothh(\Sigma^\infty_+ M)\simeq \left(\thh(M^{-\tau})\right)^\vee \simeq \Sigma^\infty_+ \mathscr{L}M.
\]
This recovers the cohomological version described in \cite[Comments 3, Section I.1.5]{stringtop}. 
\end{ex}

\begin{ex}\label{ex: compact Lie group}
Let $G$ be a compact Lie group. Then $\Sigma^\infty_+ G$ is a ring spectrum. Its Spanier-Whitehead is the Thom spectrum $G^{-\tau}$ of its stable normal bundle, just as in Example \ref{ex: dual of smooth manifold}. By Corollary \ref{cor: antiequivalence proper (co)algebras}, we obtain that $G^{-\tau}$ is an $\sph$-coalgebra, induced by the group structure on $G$. By Theorem \ref{thm most general thh cothh duality}, we obtain:
\[
\cothh(G^{-\tau})\simeq \left(\thh(\Sigma^\infty_+ G)\right)^\vee. 
\]
We also have that $\thh(\Sigma^\infty_+ G)\simeq \Sigma^\infty_+ \mathscr{L}BG$, and thus we obtain:
\[
\cothh(G^{-\tau})\simeq \left(\Sigma^\infty_+ \mathscr{L}BG\right)^\vee. 
\]
If $G$ is connected, then $BG$ is simply connected and thus by our discussion in Example \ref{ex: dual of smooth manifold}, we obtain:
\[
\cothh(G^{-\tau})\simeq \left(\cothh(\Sigma^\infty_+ BG)\right)^\vee.
\]
In \cite[5.7]{toolscothh}, the mod $p$-homology of $\cothh(\Sigma^\infty_+BG)$ was computed for $G$ equals one of the Lie groups $U(n)$, $SU(n)$, $Sp(n)$, $SO(2k)$, $G_2$, $F_4$, $E_6$, $E_7$ and $E_8$.
\end{ex}

\begin{rem} \label{rem e smash e is not a coalgebra}
Let $A\rightarrow B$ be a map of commutative $\ei$-algebras in a symmetric monoidal $\infty$-category $\C$. One might expect $B\otimes_A B$ to be an $\ee$-coalgebra in  $B$-modules in $\C$ due to the comultiplicaton map
\begin{multline*}
    B\otimes_A B \cong B\otimes_A A \otimes_A  B \to B\otimes_A B \otimes_A B \cong \\ B \otimes_A (B \otimes_B B) \otimes_A B  \cong (B\otimes_A B)\otimes_B (B\otimes_A B),
\end{multline*}
and the counit map 
\[ B\otimes_A B \to B \]
given by the multiplication on $B$. However, this does not work. This is due to the fact that the comultiplication map above is not a map of $B$-modules in general but of $B$-bimodules. The $B$-module structure on $ B\otimes_A B$ is given by the map 
\[B \cong B\otimes_A A \to B\otimes_A B,\]
but the $B$-module structure of the right hand side is given by the one induced by the tensor product in $B$-modules. We believe that there is no general way of turning this comultiplication map into a map of $B$-modules. 
\end{rem}
\section{Computations}\label{sec computations in ha modules}

In this section, we apply Theorem  \ref{thm equivalence of connective coalgebras and coconnective coalgebras in chains} to construct interesting examples of coalgebras in module spectra and then we use Theorem \ref{thm cothh duality in chains} to compute the topological coHochschild homology groups of these coalgebras. Furthermore, all our topological coHochschild homology computations are algebraic. In other words, in each case, we compute $\coTHH^{HR}(C)$ for some discrete ring $R$ and $HR$-coalgebra $C$. 

Each subsection below is devoted to the construction of a particular coalgebra and its coHochschild homology computation. In Section \ref{sec cothh of dsa in chains}, we show that $\hfp \wdg_{\hz}\hfp$ is an $\hfp$-coalgebra in a unique way and compute its coHochschild homology groups. Section \ref{sec cothh of steenrod algebra spectrum} is devoted to the study of the Steenrod algebra spectrum as an $\hfp$-coalgebra and in Section \ref{sec loop fp is a coalgebra}, we construct an $\hfp$-coalgebra structure on $\Omega \hfp$; we compute the relevant coHochschild homology groups in each of these cases.

\subsection{CoHochschild Homology of the Dual Steenrod Algebra in \texorpdfstring{$\hz$}{TEXT}-mo\-dules}\label{sec cothh of dsa in chains}
Here, we study the dual Steenrod algebra spectrum $\hfp \wdgz \hfp$ in $\hz$-modules  and its coHochschild homology. We claim in Remark \ref{rem e smash e is not a coalgebra} that such objects, i.e.\ objects of the form $B \otimes_A B$ for some map of $\ei$-ring spectra $A \to B$, do not carry a canonical $B$-coalgebra structure in general. 
In contrast, we show below that $\hfp \wdgz \hfp$ carries a unique $\hfp$-coalgebra structure. Subsequently, we compute the coHochschild homology of $\hfp \wdgz \hfp$ with this $\hfp$-coalgebra structure using Theorem \ref{thm cothh duality in chains}. 

\begin{thm} \label{thm dsa in chains has unique coalg structure}
There is a unique $\hfp$-coalgebra structure on $\hfp \wdgz \hfp$. Furthermore, this $\hfp$-coalgebra structure lifts to a cocommutative $\hfp$-coalgebra structure.
\end{thm}

 We postpone the proof of this theorem to the end of the subsection and start the computation of the coHochschild homology groups of $\hfp \wdgz \hfp$. Note the theorem above allows us to consider $\hfp \wdgz \hfp$ as an $\hfp$-coalgebra.

\begin{thm}\label{thm cothh of dsa in chains}
The coHochschild homology groups of $\hfp \wdg_{\hz}\hfp$ are given by:
\[   \cothh^{\hfp}_i(\hfp \wdgz \hfp)\cong\left\{
\begin{array}{ll}
      \prod_{n\in\mathbb{N}}\fp & i=0 \\
      \prod_{n\in\mathbb{N}}\fp& i=1 \\
      0& \textup{otherwise.} \\
\end{array} 
\right. \]
In other words, there is an equivalence of graded $\fp$-modules:
\begin{equation} \label{eq cothh of dsa in chains}
  \cothh^{\hfp}_*(\hfp \wdgz \hfp)\cong \lambdafp(x_1) \otimes \fp[[t]]
\end{equation} 
where $\lv t \rv = 0$ and $\lv x_1 \rv = 1$. 
\end{thm}

By a Tor calculation, one obtains that:
\[\pi_*(\hfp \wdgz \hfp) \cong \lambdafp(y_1)\]
where $\lv y_1 \rv = 1$, see \cite[Theorem 4.1 in \RomanNumeralCaps 4.4]{elmendorf2007rings}. In particular, $\hfp \wdgz \hfp$ satisfies the hypothesis of Theorem \ref{thm cothh duality in chains}. Indeed, $\hfp \wdgz \hfp$ is a dualizable $\hfp$-module due to Example \ref{ex dualizable in hk modules}. We obtain the following equivalence of $\hfp$-modules:
\begin{equation*}
\begin{split}
    \cothh^{\hfp}(\hfp \wdgz \hfp) \simeq \Big(\thh^{\hfp}\left((\hfp \wdgz \hfp)^{\vee}\right)\Big)^\vee. 
\end{split}
\end{equation*}
Therefore our question reduces to the calculation of the $\fp$-Hochschild homology groups of $(\hfp \wdgz \hfp)^{\vee}$. Recall that dualization in $\hfp$-modules results in dualization of graded $\fp$-modules at the level of homotopy groups which carries coproducts to products. Therefore Proposition \ref{prop yet another hh computation} below together with the equivalence above provide Theorem \ref{thm cothh of dsa in chains}. Let $X$ denote the $\hfp$-algebra $(\hfp \wdgz \hfp)^{\vee}$ for the rest of this subsection.

\begin{prop}\label{prop yet another hh computation}

There is an isomorphism of graded $\fp$-modules:
\[
\thh^{\hfp}_*((\hfp \wdgz \hfp)^{\vee}) \cong \fp[y] \otimes \lambdafp({z_{-1}})
\]
where $\lv y \rv=0$ and $\lv z_{-1} \rv = -1$.
\end{prop}

\begin{proof}

By the Ext spectral sequence calculating the homotopy groups of  mapping spectra \cite[\RomanNumeralCaps{4}.4.4.1]{elmendorf2007rings}, we have an equivalence of $\fp$-modules
\[\pi_*X \cong \lambdafp(z_{-1})\]
where $\lv z_{-1} \rv = -1$.
This is also an isomorphism of rings since there is a unique ring structure on the right hand side.

We show below in Proposition \ref{prop dga is formal} that the $\fp$-DGA $X$ is formal, i.e.\ it is quasi-isomorphic to an $\fp$-DGA with trivial differentials. Note that in the spectral sequence corresponding to a double complex with trivial vertical differentials, all the differentials are trivial on the $E^2$ page and after. This shows that the differentials in the standard spectral sequence calculating the Hochschild homology groups of $X$ are trivial after the first page.   

For this spectral sequence, we have:
\begin{equation*}
\begin{split}
    E^2_{s,t} = \HH^{\fp}_{s,t}(\lambdafp({z_{-1}}),\lambdafp({z_{-1}})) =& \Tor^{\lambdafp({z_{-1}}) \otimes \lambdafp({z_{-1}})^{op}}_{s,t}(\lambdafp({z_{-1}}),\lambdafp({z_{-1}}))\\
    &\Longrightarrow \THH^{\hfp}_{s+t}(X).
    \end{split}
    \end{equation*}
To calculate the $E^2$ page of this spectral sequence, we use the automorphism of $\lambdafp({z_{-1}}) \otimes \lambdafp({z_{-1}})$
given by:
\[z_{-1} \otimes 1 \to z_{-1} \otimes 1 \textup{\ and \ } 1 \otimes z_{-1} \to z_{-1} \otimes 1 - 1 \otimes z_{-1}.\]
Precomposing with this automorphism, the action of the second factor of $\lambdafp({z_{-1}}) \otimes \lambdafp({z_{-1}})$ on $\lambdafp({z_{-1}})$ becomes trivial, i.e.\ this action is the one obtained by the augmentation map of $\lambdafp({z_{-1}})$.  Using the K\"unneth formula \cite[Equation (2)]{bayindir2019dgaswithpolynomial}, we obtain the following:
\begin{equation*}
    \begin{split}
        E^2 \cong &\Tor^{\lambdafp({z_{-1}}) \otimes \lambdafp({z_{-1}})^{op}}(\lambdafp({z_{-1}}),\lambdafp({z_{-1}}))\\
        \cong& \Tor^{\lambdafp({z_{-1}}) }(\lambdafp({z_{-1}}),\lambdafp({z_{-1}})) \otimes\Tor^{\lambdafp({z_{-1}}) }(\fp,\fp)   \\
        \cong & \lambdafp(z_{-1}) \otimes \Gamma_{\fp}(\sigma z_{-1}).
    \end{split}
\end{equation*}
Here: $\text{deg}(\sigma z_{-1})= (1,-1)$ and  $\Gamma_{\fp}(\sigma z_{-1})$ denotes the divided power algebra on a single generator. There is an equivalence of $\fp$-modules $\Gamma_{\fp}(\sigma z_{-1}) \cong \fp[\sigma z_{-1}]$.

As mentioned before, all the differentials in this spectral sequence are trivial on the $E^2$ page and after. Therefore, there is an isomorphism of graded $\fp$-modules \begin{equation*}
\THH^{\hfp}_*(X) \cong \fp[y] \otimes \lambdafp({z_{-1}})
\end{equation*}
 where $\lv y \rv=0$. 
\end{proof}

In order to finish the proof of Theorem \ref{thm cothh of dsa in chains}, we need to prove the following proposition. Recall that an $\fp$-DGA is said to be formal if it is quasi-isomorphic to an $\fp$-DGA with trivial differentials.

\begin{prop} \label{prop dga is formal}
Every $\fp$-DGA $Z$ with homology
\[H_*Z = \lambdafp(z_{-1})\]
 is formal as an $\fp$-DGA. In particular, $X=(\hfp \wdgz \hfp)^{\vee} $ is a formal $\fp$-DGA.
\end{prop}

\begin{rem}
In \cite[proof of Theorem 1.6]{bayindir2019dgaswithpolynomial} the first author shows that every $\fp$-DGA with homology $\lambdafp(z_{-1})$ is formal as a DGA. However, we need a slightly stronger result. We need to show that such $\fp$-DGAs are formal as $\fp$-DGAs.
\end{rem}

We work on this problem in $\hfp$-algebras and use the obstruction theory of  Hopkins and Miller \cite{rezk1998notes}. This obstruction theory provides obstructions to lifting a map of monoids in the homotopy category of $\hfp$-modules to a map of $\hfp$-algebras  \cite[4.5]{johnson2014lifting}.
\begin{proof}[Proof of Proposition \ref{prop dga is formal}]

Let $Z$ also denote an $\hfp$-algebra corresponding to the $Z$ in  Proposition \ref{prop dga is formal} and let $T$ denote the $\hfp$-algebra corresponding to the formal DGA with homology $\lambdafp(z_{-1})$. 

Since the homotopy category of $\hfp$-modules is the category of graded $\fp$-modules as a monoidal category via the functor $\pi_*-$, there is an isomorphism of monoids from $Z$ to $T$ in the homotopy category of $\hfp$-modules. By Theorem 4.5 of \cite{johnson2014lifting}, the obstructions to lifting this map to a map of $\hfp$-algebras lie in the following Andr{\'e}--Quillen cohomology groups for graded associative $\fp$-algebras \cite[Section 5.2.1]{johnson2014lifting}:
\[\Der^{s+1}(\lambdafp(z_{-1}), \Omega^{s}\lambdafp(z_{-1})) \textup{\ for \ } s\geq 1.\]
Here, $\Omega^n$ shifts down a graded module $n$ times. By the identification of  Andr{\'e}--Quillen cohomology groups with Hochschild cohomology groups, one obtains that the obstructions lie in the following groups \cite[Section 5.2.1]{johnson2014lifting}: 
\begin{equation} \label{eq obstruction groups}
\Ext_{\lambdafp(z_{-1}) \otimes \lambdafp(z_{-1})^{op}}^{s+2}(\lambdafp(z_{-1}), \Omega^{s}\lambdafp(z_{-1})) \textup{\ for \ } s\geq 1.
\end{equation}
By the discussion above, it is sufficient to show that the Ext groups in \eqref{eq obstruction groups} are trivial.
Note that all the rings in sight are graded commutative and therefore we actually have $\lambdafp(z_{-1})\op = \lambdafp(z_{-1})$. There is an automorphism of
\[\lambdafp(z_{-1}) \otimes \lambdafp(z_{-1}),\]
given by: 
\[z_{-1} \otimes 1 \to z_{-1} \otimes 1 \textup{\ and \ } 1 \otimes z_{-1} \to z_{-1} \otimes 1 - 1 \otimes z_{-1}.\]
We consider the obstruction groups in \eqref{eq obstruction groups} after precomposing with this automorphism. This ensures that the action of the first factor in $\lambdafp(z_{-1}) \otimes \lambdafp(z_{-1})$ on $\lambdafp(z_{-1})$ and $\Omega^{s}\lambdafp(z_{-1})$ are the canonical non-trivial actions and the action of the second factor on $\lambdafp(z_{-1})$ and $\Omega^{s}\lambdafp(z_{-1})$ are the trivial actions given by the augmentation map of $\lambdafp(z_{-1})$.

Let
\[\varphi \co \lambdafp(z_{-1}) \longrightarrow \lambdafp(z_{-1}) \otimes \lambdafp(z_{-1})\]
denote the map of rings given by the inclusion of the second factor. By the derived extension of scalars functor induced by $\varphi$, we obtain the first isomorphism below. The second isomorphism is obtained using the fact that $\lambdafp(z_{-1}) \otimes \lambdafp(z_{-1})$ is a flat $\lambdafp(z_{-1})$-module through the action induced by $\varphi$:
\begin{equation*}
\begin{split}
&\Ext_{ \lambdafp(z_{-1})}^{s+2}(\fp, \Omega^{s}(\fp \oplus \Omega \fp))\\
    &\cong \Ext_{\lambdafp(z_{-1}) \otimes \lambdafp(z_{-1})}^{s+2}((\lambdafp(z_{-1}) \otimes \lambdafp(z_{-1}) ) \otimes_{\lambdafp(z_{-1})}^{\mathbb{L}} \fp, \Omega^{s}\lambdafp(z_{-1})) \\
    &\cong \Ext_{\lambdafp(z_{-1}) \otimes \lambdafp(z_{-1})}^{s+2}(\lambdafp(z_{-1}), \Omega^{s}\lambdafp(z_{-1})). \\
    \end{split}
\end{equation*}
Therefore it is sufficient to show that: 
\[\Ext_{ \lambdafp(z_{-1})}^{s+2}(\fp, \Omega^{s}(\fp \oplus \Omega \fp))= 0 \textup{\ for  } s \geq 1.\]
We have the following free resolution of $\fp$ in $\lambdafp(z_{-1})$-modules:
\[ \cdots \to \Omega^2 \lambdafp(z_{-1}) \to \Omega \lambdafp(z_{-1}) \to \lambdafp(z_{-1})\to \fp.\]
Applying $\Hom_{\lambdafp(z_{-1})}(-, \Omega^{s}(\fp \oplus \Omega \fp))$ to this resolution gives a resolution which is given by the following in resolution degree $s+2$:
\[\Hom_{\lambdafp(z_{-1})}(\Omega^{s+2}\lambdafp(z_{-1}), \Omega^{s}(\fp \oplus \Omega \fp)) = 0.\]
Therefore the groups in \eqref{eq obstruction groups} containing the obstructions to the formality of $Z$ are trivial.
\end{proof}
\begin{rem}\label{rem: counter-example for dual of cothh}
Our computation above shows that:
\[
\Big(\cothh^{\hfp}(C)\Big)^\vee \not\simeq \thh^{\hfp}(C^\vee).
\]
Indeed we have that $\cothh_0^{\hfp}(\hfp \wdgz \hfp)$ is an uncountably infinite dimensional $\fp$-vector space, see \eqref{eq cothh of dsa in chains}. We also know that $\thh_0^{\hfp}((\hfp \wdgz \hfp)^\vee)$ is a countably infinite dimensional $\fp$-vector space, see Proposition 
\ref{prop yet another hh computation}. Note that dualization in $\hfp$-modules  results in $\fp$-dualization at the level of homotopy groups \cite[\RomanNumeralCaps{4}.4.1]{elmendorf2007rings}. In particular, the dual of  $\cothh_0^{\hfp}(\hfp \wdgz \hfp)$ is again an uncountable dimensional vector space and therefore it is not isomorphic to $\thh_0^{\hfp}((\hfp \wdgz \hfp)^\vee)$. 
\end{rem}

\begin{proof}[Proof of Theorem \ref{thm dsa in chains has unique coalg structure}]

To prove this theorem, we apply  Theorem \ref{thm equivalence of connective coalgebras and coconnective coalgebras in chains} for the case $A= \fp$ and $\O= \assoc$. 
Recall that $\hfp$-modules are uniquely determined by their homotopy groups; this follows by the Ext spectral sequence calculating mapping spectra \cite[\RomanNumeralCaps{4}.4.4.1]{elmendorf2007rings}. Therefore, our goal is to show that there is a unique  $\hfp$-coalgebra with homotopy groups given by:
\[\lambdafp(y_1) \cong \pi_*(\hfp \wdgz \hfp).\]
It follows by the Ext spectral sequence that the dualization functor in $\hfp$-modules results in graded $\fp$-module dualization at the level of homotopy groups. In particular, we have: 
\[\pi_*((\hfp \wdgz \hfp)^\vee) \cong \lambdafp(z_{-1})\]
where $\lvert z_{-1}\rvert = -1$. Similarly, the dual of an $\hfp$-algebra with homotopy $\lambdafp(z_{-1})$ has homotopy groups given by $\lambdafp(y_{1})$. Using this together with Theorem \ref{thm equivalence of connective coalgebras and coconnective coalgebras in chains}, we obtain that the set of  equivalence classes of $\hfp$-coalgebras with homotopy groups $\lambdafp(y_{1})$ is in bijective correspondence with the set of equivalence classes of $\hfp$-algebras with homotopy groups $\lambdafp(z_{-1})$. Furthermore, observe that the $\fp$-module $\lambdafp(z_{-1})$ has a unique ring structure on it. 

Therefore to prove the first part of Theorem \ref{thm dsa in chains has unique coalg structure}, it is sufficient to show that there is a unique $\hfp$-algebra with homotopy ring $\lambdafp(z_{-1})$. This is provided by Proposition \ref{prop dga is formal}. 
Indeed, this shows that $\hfp \wdgz \hfp$ is the dual of the formal $\fp$-DGA with homology $\lambdafp(z_{-1})$. Since this formal $\fp$-DGA is commutative, we deduce that $\hfp \wdgz \hfp$ is the dual of a commutative $\hfp$-algebra. Together with Theorem \ref{thm equivalence of connective coalgebras and coconnective coalgebras in chains}, this shows that $\hfp \wdgz \hfp$ is a cocommutative $\hfp$-coalgebra.
\end{proof}

\subsection{CoHochschild Homology of the Steenrod Algebra Spectrum} \label{sec cothh of dual steenrod algebra} \label{sec cothh of steenrod algebra spectrum}

In this section, we apply Theorem \ref{thm cothh duality in chains} to  compute the coHochschild homology of the Steenrod algebra spectrum. Recall that
\[\pi_*(\hfp \wdg \hfp)\cong \mathcal{A}_*\]
is the dual Steenrod algebra. Furthermore, we have:  
\[{[\hfp \wdg \hfp, \hfp]}_{\hfp} \simeq {[\hfp,{[\hfp,\hfp]}_{\sph}]}_{\hfp} \simeq {[\hfp,\hfp]}_{\sph},\]
where $[-,-]_R$ denotes the internal hom in $R$-modules for a given commutative ring spectrum $R$. Recall that ${[\hfp,\hfp]}_{\sph}$ is the spectrum of cohomology operations on the cohomology theory defined by $\hfp$. Indeed: 
\[\mathcal{A}=\pi_*\left({[\hfp,\hfp]}_{\sph}\right)\]
is the Steenrod algebra. 
There is an $\hfp$-coalgebra structure on ${[\hfp,\hfp]}_{\sph}$ given by Theorem \ref{thm equivalence of connective coalgebras and coconnective coalgebras in chains} and the fact that it is the dual of the $\hfp$-algebra $\hfp \wdg \hfp$. This induces the usual comultiplication on the Steenrod algebra \cite[Section \RomanNumeralCaps{2}.10.2]{schwede-book}. 
We use Theorem \ref{thm cothh duality in chains} to compute the coHochschild homology of ${[\hfp,\hfp]}_{\sph}$ with this $\hfp$-coalgebra structure. Let $(-)^\vee$ denote the linear dual functor in $\hfp$-modules. There are equivalences of $\hfp$-algebras
\[\left({[\hfp,\hfp]}_{\sph}\right)^\vee \simeq (\hfp \wdg \hfp)^{\vee\vee}\simeq \hfp \wdg \hfp\]
where the first equivalence follows by our definition of the $\hfp$-coalgebra structure on ${[\hfp,\hfp]}_{\sph}$ and the second equality follows by Theorem \ref{thm equivalence of connective coalgebras and coconnective coalgebras in chains}. 
Using Theorem \ref{thm cothh duality in chains}, we obtain: 
 \[\cothh^{\hfp}_*\left({[\hfp,\hfp]}_{\sph}\right) \cong \left(\thh^{\hfp}_*(\hfp \wdg \hfp)\right)^\vee,\]
 where $(-)^\vee$ here denotes the linear dual functor in graded $\fp$-modules.
 To compute the right hand side of this equality, we use the following equalities
 \[\thh^{\hfp}(\hfp \wdg \hfp)\simeq \hfp \wdg \thh(\hfp) \simeq (\hfp \wdg \hfp) \wdg_{\hfp} \thh(\hfp)\]
 and the B\"okstedt periodicity \cite{bokstedt1985topological} \[\thh_*(\hfp) = \fp[x_2].\]
This shows that:
\[\cothh^{\hfp}_*\left({[\hfp,\hfp]}_{\sph}\right) \cong (\mathcal{A}_* \ot \fp[x_2])^\vee,\]
where $(-)^\vee$ on the right hand side denotes the dualization functor in graded $\fp$-modules. Note that $(M^\vee)_i = (M_{-i})^\vee$ for every graded $\fp$-module $M$. The right hand side of the isomorphism above is a finite dimensional $\fp$-module at each degree and $\fp$-dualization is symmetric monoidal on finite dimensional $\fp$-vector spaces, see Remark \ref{rem dualizable in k modules are finite dimensional}. We therefore have the following isomorphisms of graded $\fp$-modules:
\[ (\mathcal{A}_* \ot \fp[x_2])^\vee\cong (\mathcal{A}_*)^\vee \otimes \fp[x_2]^\vee\cong \mathcal{A} \otimes \fp[x_{-2}],\]
where $\lv x_{-2}\rv = -2$. We obtain the following result.

\begin{thm}\label{thm: cothh of steenrod algebra}
There is an equivalence of graded $\fp$-modules:
\[\cothh^{\hfp}_*\left({[\hfp,\hfp]}_{\sph}\right) \cong \mathcal{A} \otimes \fp[x_{-2}],\]
where $\mathcal{A}$ denotes the Steenrod algebra and $\lv x_{-2}\rv = -2$.
\end{thm}

\subsection{An Interesting Coalgebra in \texorpdfstring{$\hz$}{TEXT}-modules} \label{sec loop fp is a coalgebra}

Recall that dualizable $\hz$-modules are precisely those that can be obtained from $\hz$ via finitely many cofiber sequences and retracts \cite[3.2.3]{lurie2}. Therefore, the $H\z$-algebra $H\fp$ is dualizable as an $H\z$-module. This is due to the cofiber sequence
\[H\z \xrightarrow{\cdot p}H\z \to H\fp.\]
Thus $H\fp^\vee$, the dual of $\hfp$ in $\hz$-modules, is an $H\z$-coalgebra, see Corollary \ref{cor: antiequivalence proper (co)algebras}. Applying the  dualization functor to the sequence above, one obtains a fiber sequence
\[\Sigma (\hfp^\vee) \leftarrow H\z \xleftarrow{\cdot p} H\z \xleftarrow{} \hfp^\vee, \]
which shows that $\Sigma (\hfp^\vee) = \hfp$. Therefore we have an equivalence 
\[H\fp^\vee \simeq \Omega H\fp,\]
of $H\z$-modules. We consider $\Omega \hfp$ as an $\hz$-coalgebra through this equivalence.

\begin{thm}\label{thm: loop-hfp}
With the $H\z$-coalgebra structure on $\Omega\hfp$ described above, the coHochschild homology of $\Omega\hfp$ is given by:
\[\cothh_*^{H\z}(\Omega H\fp)\cong \Omega \fp[x_{-2}],\]
as a graded $\fp$-module where $\lv x_{-2} \rv= -2$ and $\Omega$ on the right hand side denotes the functor that decreases the grading by 1.
\end{thm}

\begin{proof}
We compute the coHochschild homology of $\Omega H\fp$ using Theorem \ref{thm cothh duality in chains}. Since $\hfp$ is dualizable in $\hz$-modules, we have the following equivalences of $\hz$-algebras:
 \[(\Omega \hfp)^\vee \simeq \hfp^{\vee \vee} \simeq \hfp,\]
 where the first equivalence follows by our definition of $\Omega \hfp$ as an $\hz$-coalgebra and the second equivalence follows by Corollary \ref{cor: antiequivalence proper (co)algebras}. Using Theorem \ref{thm cothh duality in chains}, we obtain:
\[\cothh_*^{H\z}(\Omega H\fp)= \pi_*( (\thh^{H\z}(H\fp))^\vee).\]
There is an equivalence of graded $\fp$-modules
\[\pi_*(\thh^{H\z}(H\fp))\cong \Gamma_{\fp}(x_2)\cong \fp[x_2].\]
Here: $\Gamma_{\fp}(x_2)$ denotes the divided power algebra over $\fp$ on a single generator where $\lv x_2\rv=2$. Hochschild homology groups above are obtained via the K\"unneth spectral sequence calculating
\[\pi_*(H\fp \wdg_{H\fp \wdgz \hfp}\hfp),\]
whose $E^2$ page, as a bigraded $\fp$-module, is given by: 
\[\Tor_{*,*}^{\Lambda_{\fp}(y_1)}(\fp,\fp) \cong \Gamma_{\fp}(x_{1,1})\cong \fp[x_{1,1}].\]
Since equivalence type of Eilenberg-Mac Lane spectra are uniquely determined by their homotopy groups, we have an equivalence of $\hz$-modules:
\[\thh^{\hz}(\hfp) \simeq \bigvee_{i= 0}^\infty \Sigma^{2i} \hfp.\]
Therefore:
\begin{equation*}
\begin{split}
    \thh^{\hz}(\hfp)^\vee \simeq& [\bigvee_{i= 0}^\infty  \Sigma^{2i} \hfp,\hz]   \simeq  \prod_{i=0}^\infty \Omega^{2i}[\hfp,\hz] \\
     \simeq & \prod_{i=0}^\infty \Omega^{2i}\hfp^\vee \simeq \prod_{i=0}^\infty \Omega^{2i} \Omega \hfp.
\end{split}
\end{equation*}
We obtain: 
\[\cothh_*^{H\z}(\Omega H\fp)\cong \Omega \fp[x_{-2}],\]
where $\lv x_{-2} \rv= -2$ and $\Omega$ denotes the functor that decreases the grading by 1.
\end{proof}

\renewcommand{\bibname}{References}
\bibliographystyle{amsalpha}
\bibliography{biblio}
\end{document}